\documentclass[11pt,twoside]{preprint}

\usepackage{times}
\usepackage{amssymb,amsmath, amsfonts,amsthm}
\usepackage{shuffle}
\usepackage{mhequ}
\usepackage{appendix}
\usepackage{color}
\usepackage{hyperref}
\usepackage{breakurl}
\usepackage{mathtools}
\usepackage{microtype}
\usepackage{enumitem}

\theoremstyle{plain}
\newtheorem{theorem}{Theorem}[section]

\newtheorem{proposition}[theorem]{Proposition}
\newtheorem{lemma}[theorem]{Lemma}
\newtheorem{assumption}{Assumption}
\theoremstyle{definition}
\newtheorem{definition}[theorem]{Definition}
\theoremstyle{remark}
\newtheorem{remark}[theorem]{Remark}

\newcommand{\VERT}{\vert\kern-0.4ex\vert\kern-0.4ex\vert}
\newcommand{\angles}[2]{\langle #1, #2 \rangle}
\newcommand{\real}{{\mathbb R}}
\newcommand{\com}{{\mathbb C}}

\newcommand{\intgr}{{\mathbb Z}}
\def\BV{\mathrm{BV}}

\def\Xint#1{\mathchoice
	{\XXint\displaystyle\textstyle{#1}}%
	{\XXint\textstyle\scriptstyle{#1}}%
	{\XXint\scriptstyle\scriptscriptstyle{#1}}%
	{\XXint\scriptscriptstyle\scriptscriptstyle{#1}}%
\!\int}
\def\XXint#1#2#3{{\setbox0=\hbox{$#1{#2#3}{\int}$ } \vcenter{\hbox{$#2#3$ }}\kern-.56\wd0}}
\def\dashint{\Xint-}

\definecolor{darkred}{rgb}{0.9,0.1,0.1}
\definecolor{darkblue}{rgb}{0,0,0.7}
\definecolor{darkgreen}{rgb}{0,0.5,0}

\let\eps\varepsilon

\begin{document}

\title{Optimal rate of convergence for stochastic \\ Burgers-type equations}
\author{M. Hairer, K. Matetski}
\institute{Mathematics Department, University of Warwick}
\maketitle

\begin{abstract}
Recently, a solution theory for one-dimensional stochastic PDEs of Burgers type driven by space-time white noise was developed. In particular, it was shown that natural numerical approximations of these equations converge and that their convergence rate in the uniform topology (in probability) is arbitrarily close to $\frac{1}{6}$. In the present article we improve this result in the case of additive noise by proving that the optimal rate of convergence is arbitrarily close to $\frac{1}{2}$.
\end{abstract}


\section{Introduction}
The goal of this article is to study numerical approximations of stochastic PDEs of Burgers type on the circle $\mathbb{T} = \real / (2\pi \intgr)$ given by
\begin{equ}[eq:burgers]
\begin{array}{rcl}
du = \left[\nu \Delta u + F(u) + G(u) \partial_{x} u \right] dt + \sigma d W(t)\;, \qquad u(0) = u^{0}\;.
\end{array}
\end{equ}
Here, $u : \real_+ \times \mathbb{T} \times \Omega \rightarrow \real^n$, where $\left( \Omega, \mathcal{F}, \mathbb{P} \right)$ is a probability space, $\Delta=\partial^2_x$ is the Laplace operator on the circle $\mathbb{T}$, the derivative $\partial_x$ is understood in the sense of distributions, the function $F:\real^n \rightarrow 
\real^n$ is of class $\mathcal{C}^1$, the function $G:\real^n \rightarrow \real^{n \times n}$ is of class $\mathcal{C}^{\infty}$, and $\nu, \sigma \in \real_+$ are positive constants. Finally, $W$ is an $L^2$-cylindrical Wiener process \cite{DPZ}, i.e.\ equation \eqref{eq:burgers} is driven by space-time white noise. The product appearing in the term 
$G(u) \partial_{x} u$ is matrix-vector multiplication.

The difficulty in dealing with \eqref{eq:burgers} comes from the nonlinearity $G(u) \partial_{x} u$ and is caused by the low space-time regularity of the driving noise. Indeed, it is well-known that the pairing 
\begin{equ}
\mathcal{C}^{\alpha} \times \mathcal{C}^\beta \ni (v,u) \mapsto v\, \partial_x u
\end{equ}
is well defined if and only if $\alpha + \beta > 1$ (see Appendix \ref{sec:distributions} and \cite{BCD11}). On the other hand, 
one expects solutions to \eqref{eq:burgers} to have 
the spatial regularity of the solution of the linearised equation
\begin{equ}[eq:linearised]
dX(t) = \nu \Delta X dt + \sigma d W(t)\;.
\end{equ}
For any fixed time $t > 0$, the solution to the stochastic heat equation \eqref{eq:linearised} 
has almost surely H\"{o}lder regularity $\alpha < \frac{1}{2}$, 
but is \textit{not} $\frac{1}{2}$-H\"older continuous (see \cite{Walsh,DPZ,Hai09}). This implies in particular that the product $G(X) \partial_x X$ is not well-defined in this case, and it is not a priori clear how to define a solution to the equation \eqref{eq:burgers}.

In the case $G \equiv 0$ this problem does of course not occur. Equations of this type and their numerical approximations were well studied and the results can be found in \cite{Gyo98b, Gyo99}. Moreover, it was shown in \cite{DG01} that the optimal rate of uniform convergence in this case is $\frac{1}{2} - \kappa$, for every $\kappa > 0$, as the spatial discretisation tends to zero.

For non-zero $G$, the difficulty can be easily overcome in the gradient case, i.e.\ when $G = \nabla \mathcal{G}$ for some smooth function $\mathcal{G} : \real^n \rightarrow \real^n$. In this case, postulating the chain rule, the nonlinear term can be rewritten as
\begin{equ}[e:chainRule]
G\left(u(t,x)\right) \partial_{x} u(t,x) = \partial_{x} \mathcal{G}\left(u(t,x)\right),
\end{equ}
which is a well-defined distribution as soon as $u$ is continuous. 
The existence and uniqueness results in the gradient case can be found in \cite{Gyo98a, DPDT94}. In the 
article \cite{AG06}, the finite difference scheme was studied for the case $G(u)=u$, and $L^2$-convergence was  shown with rate $\gamma$, for every $\gamma < \frac{1}{2}$. The same rate of convergence was obtained in \cite{BJ13} in the $L^\infty$ topology for Galerkin approximations.

For a general sufficiently smooth function $G$, a notion of solution was given in \cite{Hai11a}. The key idea of the approach was to test the nonlinearity with a smooth test function $\varphi$ and to formally rewrite it as
\begin{equ}[eq:test_nonlin]
\int_{-\pi}^{\pi} \varphi(x) G\left(u(t,x)\right) \partial_{x} u(t,x)\, dx = \int_{-\pi}^{\pi} \varphi(x) G\left(u(t,x)\right) d_x u(t,x)\;.
\end{equ}
As it was stated above, we expect $u$ to behave locally like the solution to the linearised equation 
\eqref{eq:linearised}. It was shown in \cite{Hai11a} that the latter can be viewed in a canonical way as
a process with values in a space of rough paths.
This correctly suggests that the theory of controlled rough paths \cite{Gub04, Gub10} 
could be used to deal with the integral \eqref{eq:test_nonlin} in the pathwise sense. 
The quantity \eqref{eq:test_nonlin} is uniquely defined up to a choice of the iterated 
integral which represents the integral of $u$ with respect to itself. This implies that for 
different choices of the iterated integral we obtain different solutions, which is similar to 
the choice between It\^{o} and Stratonovich stochastic integrals in the theory of SDEs.
In the present situation however, there is a unique choice for the iterated integral 
which respects the symmetry of
the linearised equation under the substitution $x \mapsto -x$, and this corresponds to the
``Stratonovich solution''. This natural choice is also the one for which 
the chain rule \eqref{e:chainRule} holds in the particular case when $G$ is a gradient.

Using the rough path approach, numerical approximations to \eqref{eq:burgers} in the gradient case 
without using the chain rule were studied in \cite{HM12}. It was shown that the corresponding approximate solutions converge in suitable Sobolev spaces to a 
limit which solves \eqref{eq:burgers} with an additional correction term, which can be computed 
explicitly. This term is an analogue to the It\^{o}-Stratonovich correction term in the classical 
theory of SDEs.

In \cite{HW10}, the solution theory was extended to Burgers-type equations with 
multiplicative noise (i.e.\ when the multiplier of the noise term is a nonlinear 
local function $\theta(u)$ of the solution). 
Analysis of numerical schemes approximating the equation in the multiplicative case was performed in 
\cite{HMW14}, where the appearance of a correction term was observed and the rate of convergence in 
the uniform topology was shown to be of order $\frac{1}{6}-\kappa$, for every $\kappa > 0$.

In this article, we prove that in the case of additive noise the rate of convergence in 
the supremum norm is $\frac{1}{2}-\kappa$, for every $\kappa > 0$. Actually, it turns out 
to be technically advantageous to consider convergence in H\"{o}lder spaces with
H\"{o}lder exponent very close to zero. The 
main difference to \cite{HMW14} is that we cannot use the classical theory of controlled 
rough paths which applies only in the H\"{o}lder spaces of regularity from 
$\left(\frac{1}{3}, \frac{1}{2}\right]$, to approximate the rough integral 
\eqref{eq:test_nonlin}. To show the convergence in the H\"{o}lder spaces of lower 
regularity, we use the results from \cite{Gub10}, which generalize the theory of controlled rough paths for 
 functions of any positive regularity.


\subsection{Assumptions and statement of the main result}

As before we assume that $F \in \mathcal{C}^1$ and $G \in \mathcal{C}^{\infty}$ in 
\eqref{eq:burgers}. For $\varepsilon > 0$ we consider the approximate stochastic PDEs on the circle $\mathbb{T}$ given by
\begin{equ}[eq:SPDE_approx]
\begin{array}{rcl}
d u_{\varepsilon} = \left[\nu \Delta_{\varepsilon} u_{\varepsilon} + F(u_{\varepsilon}) + G(u_{\varepsilon}) D_{\varepsilon} u_{\varepsilon} \right] dt + \sigma H_{\varepsilon}dW\;, \quad u_{\varepsilon}(0) = u^{0}_{\varepsilon}\;.
\end{array}
\end{equ}
Here, the operators $\Delta_{\varepsilon}$, $D_{\varepsilon}$ and $H_{\varepsilon}$ are defined as Fourier multipliers providing approximations of $\Delta$, $\partial_{x}$ and the identity operator respectively, and are given by
\begin{equs}
\widehat{\Delta_{\varepsilon} u}(k) = -k^{2} m(\varepsilon k) \widehat{u}(k)\;, \enspace \widehat{D_{\varepsilon} u}(k) = i k g(\varepsilon k) \widehat{u}(k)\;, \enspace \widehat{H_{\varepsilon} W}(k) = h(\varepsilon k) \widehat{W}(k)\;.
\end{equs}
Below we provide the assumptions on the functions $m$, $g$ and $h$. We start with the assumptions on $m$.

\begin{assumption}
\label{as:first}
\label{as:second}
The function $m: \real \rightarrow (0,\infty]$ is even, satisfies $f(0)=1$, is continuously differentiable on the interval $[-\delta, \delta]$ for some $\delta > 0$, and there exists $c_{m} \in (0,1)$ such that $m \geq c_{m}$.

Furthermore, the functions $b_t$ given by $b_{t}(x) := \exp\left( -x^{2}m(x)t \right)$
are uniformly bounded in $t > 0$ in the bounded variation norm, i.e. $\sup_{t > 0} |b_{t}|_{\BV} < \infty$.
\end{assumption}

Our next assumption concerns $g$, which defines the approximation to the spatial derivative.

\begin{assumption}
\label{as:third}
There exists a signed Borel measure $\mu$ on $\real$ such that
$\int_{\real} e^{ikx} \mu(dx) = ik g(k)$, and such that
\begin{equ}
\mu(\real)=0\;,\qquad|\mu|(\real) < \infty\;,\qquad\int_{\real} x \mu(dx) = 1\;.
\end{equ}
Moreover, the measure $\mu$ has all finite moments, i.e.
$\int_{\real} |x|^{k} |\mu|(dx) < \infty$,
for any integer $k \geq 1$. 
\end{assumption}

In particular, the approximate derivative can be expressed as 
\[
\left(D_{\varepsilon}u\right)(x)= \frac{1}{\varepsilon} \int_{\real} u(x+\varepsilon y) \mu(dy)\;, 
\]
where we identify $u: \mathbb{T} \rightarrow \real$ with its periodic extension to all $\real$. Our last assumption is on the function $h$, which defines the approximation of noise.

\begin{assumption}
\label{as:fourth}
The function $h$ is even, bounded, and such that $h^{2}/m$ and $h/(m + 1)$ are of bounded variation. Furthermore, $h$ is twice differentiable at the origin with $h(0)=1$ and $h'(0)=0$.
\end{assumption}

The difference with the assumptions in \cite{HMW14} is that we require in Assumption \ref{as:third} all the moments of the measure $\mu$ to be finite and in Assumption \ref{as:fourth} the function $h/(m + 1)$ to be of bounded variation. We use the latter assumption in Lemma \ref{lem:X_X_eps} in order to use the bounds on lifted rough paths obtained in \cite{FGGR13}. All the examples of approximations provided in \cite{HM12} (including finite difference schemes) still satisfy our assumptions.

Let $\bar{u}$ be the solution to the modified equation \eqref{eq:burgers},
\begin{equ}[eq:burgers_modif]
d \bar{u} = \left[\nu \Delta \bar{u} + \bar{F}(\bar{u}) + G(\bar{u}) \partial_{x} \bar{u} \right] dt + \sigma dW\;,\qquad
\bar{u}(0) = u^{0}\;,
\end{equ}
where, for $i = 1, \ldots, n$, the modified reaction term is given by
\begin{equ}
\bar{F}_{i} := F_{i} - \Lambda ~\mathrm{div} G_{i}\;.
\end{equ}
Here, we denote by $G_i$ the $i$th row of the matrix-valued function $G$, and the correction constant is defined by
\begin{equ}
\Lambda := \frac{\sigma^2}{2\pi \nu} \int_{\real_{+}} \int_{\real} \frac{(1-\cos(yt)) h^{2}(t)}{t^{2} m(t)} \mu(dy) dt\;.
\end{equ}
It follows from the assumptions that $\Lambda$ is well-defined. In fact, the Assumption \ref{as:fourth} says that $|h^2 / m|$ is bounded, and by the Assumption \ref{as:third} the measure $\mu$ has a finite second moment, what yields the existence of $\Lambda$.

As we do not assume boundedness of the functions $F$ and $G$, and their derivatives, the solution 
can blow up in finite time. To overcome this difficulty we consider solutions 
only up to some stopping times. More precisely, for any $K > 0$ we define the stopping times
\begin{equ}
\tau^*_K := \inf \{ t > 0: \Vert \bar{u}(t)\Vert_{\mathcal{C}^0} \geq K \}\;,
\end{equ}
where $\Vert \cdot \Vert_{\mathcal{C}^{0}}$ is the supremum norm. The blow-up time of $\bar{u}$ is then defined as $\tau^* := \lim_{K \uparrow \infty} \tau^*_K$ in probability.

Our main theorem gives the convergence rate of the solutions of the approximate equations \eqref{eq:SPDE_approx} to the solution of the modified equation \eqref{eq:burgers_modif}.

\begin{theorem}
\label{thm:first}
Let for every $\eta \in \bigl(0,\frac{1}{2}\bigr]$ the initial values satisfy
\begin{equ}
\mathbb{E} \Vert u^{0}\Vert_{\mathcal{C}^{\eta}} < \infty\;, ~~~\sup_{0 < \varepsilon \leq 1} \mathbb{E} \Vert u_{\varepsilon}^{0}\Vert_{\mathcal{C}^{\eta}} < \infty\;.
\end{equ}
Then, there exists $\alpha_0 > 0$ such that if, for some $\alpha \in (0,\alpha_0]$ and some constant $C > 0$ independent of $\eps$, one has
\begin{equ}
\mathbb{E} \Vert u^{0} - u_{\varepsilon}^{0} \Vert_{\mathcal{C}^{\alpha}} \leq C\varepsilon^{\frac{1}{2} - \alpha}\;,
\end{equ}
then there exists a family of stopping times $\tau_{\varepsilon}$ satisfying $\lim_{\varepsilon \downarrow 0} \tau_{\varepsilon} = \tau^*$ in probability such that
\begin{equ}
\lim_{\varepsilon \downarrow 0} \mathbb{P}\Big[\sup_{t \in [0, \tau_{\varepsilon}]} \Vert \bar{u}(t) - u_{\varepsilon}(t) \Vert_{\mathcal{C}^{0}} \geq \varepsilon^{\frac{1}{2} - \alpha}\Big] = 0\;.
\end{equ}
\end{theorem}

\begin{remark}
The rate of convergence obtained in \cite{HMW14} was ``almost" $\frac{1}{6}$, in the sense that it is $\frac{1}{6} - \kappa$ for any $\kappa > 0$. To improve this result we consider convergence of the solutions in the H\"{o}lder spaces of the regularities close to zero. This approach creates difficulties when working with the rough integrals \eqref{eq:test_nonlin}. In fact, the bounds on the rough integrals, in particular in \cite[Lemma 5.3]{HMW14}, hold only in the H\"{o}lder spaces $\mathcal{C}^\alpha$ with $\alpha \in \left(\frac{1}{3}, \frac{1}{2}\right)$ and the norms explode as $\alpha$ approaches $\frac{1}{3}$. To have reasonable bounds in the H\"{o}lder spaces of lower regularity, we have to include into the definition of the rough integrals the iterated integrals of the controlling process $X$ of higher order. In \cite{HMW14} it was enough to consider only the iterated integrals of order two. In particular, the smaller $\alpha$ is in Theorem~\ref{thm:first}, the more iterated integrals we have to consider to define the rough integral \eqref{eq:test_nonlin} (see Section~\ref{sec:RP} for more details).

If the function $G$ is only of class $\mathcal{C}^p$ for some $p \geq 3$, we can consider the iterated integrals of $X$ only up to the order $p-1$ (see Subsection \ref{ssec:integrals}). As a consequence, the argument in the proof of Theorem \ref{thm:first} gives the rate of convergence only ``almost" $\frac{1}{2} - \frac{1}{p}$. This is precisely the rate of convergence obtained in \cite{HMW14}, where $p$ was taken to be $3$.
\end{remark}

\begin{remark}
If the functions $\bar F$ and $G$ are bounded together will all of their derivatives, the solution $\bar u$ is global, i.e. $\tau^* = +\infty$ a.s., see \cite[Thm~3.6]{Hai11a}. 
In this case, the argument of the proof of Theorem~\ref{thm:first} shows that 
one can take for example $\tau_\varepsilon = T$ a.s. for any fixed time $T>0$. Since it is straightforward in this
case to obtain uniform bounds (on finite time intervals) on the $p$th moment of the solution for every $p$,
uniformly over $\eps \in (0,1]$, this implies strong $L^p$-convergence of the approximate solutions on every 
bounded time interval. In general, we do expect to have $\tau^* < \infty$, which of course precludes any
form of $L^p$ convergence without cutoffs.
\end{remark}

\begin{remark}
By changing the time variable and the functions in \eqref{eq:burgers} by a constant multiplier, we can obtain an equivalent equation with $\nu = 1$. Moreover, we can assume $\sigma = 1$. In what follows we only consider these values of the constants.
\end{remark}


\subsection{Structure of the article}
In Section \ref{sec:RP} we review the theories of rough paths and controlled rough paths. Section \ref{sec:notion} is devoted to the results obtained in \cite{Hai11a}. In particular, here we provide a notion of solution and the existence and uniqueness results for the Burgers type equations with additive noise. In Section \ref{sec:approx} we define the rough integrals and formulate the mild solution to the approximate equation \eqref{eq:SPDE_approx} in a way appropriate for working in the H\"{o}lder spaces of low regularity. The proof of Theorem \ref{thm:first} is provided in Section \ref{sec:convergence}. The following sections give bounds on the corresponding terms in the equations \eqref{eq:burgers_modif} and \eqref{eq:SPDE_approx}: in Sections \ref{sec:reaction} and \ref{sec:correction} we consider the reaction terms and Section \ref{sec:rough_int} is devoted to the terms involving the rough integrals. In Appendix \ref{sec:distributions} we prove a Kolmogorov-like criterion for distribution-valued processes. Appendix \ref{appx:heat_sg} provides regularity properties of the heat semigroup and its approximate counterpart on the H\"{o}lder spaces.


\subsection{Spaces, norms and notation}

Throughout this article, we denote  by $\mathcal{C}^0$  the space of continuous functions on the circle $\mathbb{T}$ endowed with the supremum norm. 

For functions $X: \real \rightarrow \real^n$ (or $\real^{n \times n}$) and $R: \real^2 \rightarrow \real^n$ (or $\real^{n \times n}$), such that $R$ vanishes on the diagonal, we define respectively H\"{o}lder seminorms with a given parameter $\alpha \in (0,1)$:
\begin{equ}
\Vert X \Vert_{\alpha} := \sup_{x \ne y} \frac{|X(x) - X(y)|}{|x - y|^{\alpha}}\;, \qquad \Vert R \Vert_{\alpha} := \sup_{x \ne y} \frac{|R(x,y)|}{|x - y|^{\alpha}}\;.
\end{equ}
By $\mathcal{C}^\alpha$ and $\mathcal{B}^\alpha$ respectively we denote the spaces of functions for which these seminorms are finite. Then $\mathcal{C}^\alpha$ endowed with the norm $\Vert \cdot \Vert_{\mathcal{C}^\alpha} = \Vert \cdot \Vert_{\mathcal{C}^0} + \Vert \cdot \Vert_{\alpha}$ is a Banach space. $\mathcal{B}^\alpha$ is a Banach space endowed with $\Vert \cdot \Vert_{\mathcal{B}^\alpha} = \Vert \cdot \Vert_{\alpha}$. 

The H\"{o}lder space $\mathcal{C}^\alpha$ of regularity $\alpha \geq 1$ consists of $\lfloor \alpha \rfloor$ times continuously differentiable functions whose $\lfloor \alpha \rfloor$-th derivative is $(\alpha - \lfloor \alpha \rfloor)$-H\"{o}lder continuous.
For $\alpha < 0$ we denote by $\mathcal{C}^\alpha$ the Besov space $\mathcal{B}^\alpha_{\infty, \infty}$ (see Appendix~\ref{sec:distributions} for the definition).

We also define space-time H\"{o}lder norms, i.e. for some $T > 0$ and functions $X:[0,T] \times \mathbb{T} \rightarrow \real^n$ (or $\real^{n \times n}$) and $R:[0,T] \times \mathbb{T}^2 \rightarrow \real^n$ (or $\real^{n \times n}$), any $\alpha \in \real$ and any $\beta > 0$ we define
\begin{equ}[eq:C_B_norms]
\Vert X \Vert_{\mathcal{C}^{\alpha}_{T}} := \sup_{s \in [0,T]} \Vert X(s) \Vert_{\mathcal{C}^{\alpha}}\;,\quad \Vert R \Vert_{\mathcal{B}^{\beta}_{T}} := \sup_{s \in [0,T]} \Vert R(s) \Vert_{\mathcal{B}^{\beta}}\;.
\end{equ}
We denote by $\mathcal{C}^{\alpha}_{T}$ and $\mathcal{B}^{\alpha}_{T}$ respectively the spaces 
of functions/distributions for which the norms \eqref{eq:C_B_norms} are finite. Furthermore, 
in order to deal with functions $X$ exhibiting a blow-up with rate $\eta >0$ near $t=0$, we define the norm
\[
\Vert X \Vert_{\mathcal{C}^{\alpha}_{\eta, T}} := \sup_{s \in (0,T]} s^{\eta} \Vert X(s) \Vert_{\mathcal{C}^{\alpha}}\;.
\]
Similarly to above, we denote by $\mathcal{C}^{\alpha}_{\eta, T}$ the space of functions/distributions 
for which this norm is finite.

By $\Vert \cdot \Vert_{\mathcal{C}^\alpha \rightarrow \mathcal{C}^\beta}$ we denote the operator norm of a linear map acting from the space $\mathcal{C}^\alpha$ to $\mathcal{C}^\beta$.
When we write $x \lesssim y$, we mean that there is a constant $C$, independent of the relevant quantities, such that $x \leq Cy$.

\subsection*{Acknowledgements}

We would like to thank H. Weber for numerous discussions of this and related problems.
MH's research was funded by the Philip Leverhulme trust through a leadership award, 
by the Royal Society through a research merit award, and by the ERC through a consolidator award.


\section{Elements of rough path theory}
\label{sec:RP}

In this section we provide an overview of rough path theory and controlled rough paths. For more information on rough paths theory we refer to the original article \cite{Lyo98} and to the monographs \cite{LQ02,LCL07,FV10,Book}.

One of the aims of rough paths theory is to provide a consistent and robust way
of defining the integral
\begin{equ}[eq:integral_Y_X]
\int_s^t Y(r) \otimes dX(r)\;,
\end{equ}
for processes $Y, X \in \mathcal{C}^\alpha$ with any H\"{o}lder exponent $\alpha \in \left(0,\frac{1}{2}\right]$. If $\alpha > \frac{1}{2}$, then the integral can be defined in Young's sense \cite{You36} as the limit of Riemann sums. If $\alpha \leq \frac{1}{2}$, 
however, the Riemann sums may diverge (or fail to converge to a limit independent of the partition) and the integral cannot be defined in this way. Given $X \in \mathcal{C}^\alpha$ with
$\alpha \in \left(0, \frac{1}{2}\right]$, the theory of 
(controlled) rough paths allows to define \eqref{eq:integral_Y_X} in a consistent way
for a certain class of integrands $Y$. To this end however, one has to consider not 
only the processes $X$ and $Y$, but suitable additional ``higher order'' information.

We fix $0 < \alpha \leq \frac{1}{2}$ and $p = \lfloor 1 / \alpha \rfloor$ to be the largest integer such that $p \alpha \leq 1$. We then define the $p$-step truncated tensor algebra
\begin{equ}
T^{(p)}\big(\real^n\big) := \bigoplus_{k=0}^{p} \big(\real^n\big)^{\otimes k}\;,
\end{equ}
whose basis elements can be labelled by words of length not exceeding $p$ 
(including the empty word), based on the 
alphabet $\mathcal{A} = \{1, \ldots, n\}$. We denote this set of words by $\mathcal{A}_p$. Then the correspondence $\mathcal{A}_p \to T^{(p)}(\real^n)$ is given by $w \mapsto e_w$ with $e_w = e_{w_1} \otimes \ldots \otimes e_{w_k}$, for $w = w_1 \ldots w_k$ and $e_{\emptyset} = 1 \in \big(\real^n\big)^{\otimes 0}\approx \real$, where $\{e_{i}\}_{i \in \mathcal{A}}$ is the canonical basis of $\real^n$.

There is an operation $\shuffle$, called shuffle product \cite{Reu93}, defined on the free algebra generated by $\mathcal{A}$. For any two words the shuffle product gives all the possible ways of interleaving them in the ways that preserve the original order of the letters. For example, if $a$, $b$ and $c$ are letters from $\mathcal{A}$, then one has the identity
\begin{equ}
ab \shuffle ac = abac + 2aabc + 2aacb + acab\;.
\end{equ}
We also define both the shuffle and the concatenation product of two elements from $T^{(p)}\big(\real^n\big)$, i.e. for any two words $w, \bar{w} \in \mathcal{A}_p$ we define
\begin{equ}
e_{w} \shuffle e_{\bar{w}} := e_{w \shuffle \bar{w}}\;,\qquad
e_{w} \otimes e_{\bar{w}} := e_{w \bar{w}}\;,
\end{equ}
if the sums of the lengths of the two words do not exceed $p$ and $e_{w} \shuffle e_{\bar{w}} = e_w \otimes e_{\bar w} = 0$ otherwise. This is extended to all of $T^{(p)}\big(\real^n\big)$ by linearity.
With these notations at hand, we give the following definition:

\begin{definition}
\label{def:RP}
A {\it geometric rough path} of regularity $\alpha \in \left(0, \frac{1}{2}\right]$ is a map $\mathbf{X} : \real^2 \rightarrow T^{(p)}\big(\real^n\big)$, where as above $p = \lfloor 1 / \alpha \rfloor$, such that
\begin{enumerate}
\item $\angles{\mathbf{X}(s,t)}{e_w \shuffle e_{\bar w}} = \angles{\mathbf{X}(s,t)}{e_w} \angles{\mathbf{X}(s,t)}{e_{\bar w}}$, for any $w, \bar w \in \mathcal{A}_p$ with $|w| + |\bar w| \le p$,
\item $\mathbf{X}(s,t) = \mathbf{X}(s,u) \otimes \mathbf{X}(u,t)$, for any $s, u, t \in \real$,
\item $\Vert \angles{\mathbf{X}}{e_w} \Vert_{\mathcal{B}^{\alpha |w|}} < \infty$, for any word $w \in \mathcal{A}_p$ of length $|w|$.
\end{enumerate}
\end{definition}

If we define $X^i(t) := \angles{\mathbf{X}(0,t)}{e_i}$ for any $i \in \mathcal{A}$, then the components of $\mathbf{X}(s,t)$ of higher order should be thought of as defining the iterated integrals
\begin{equ}[eq:iter_int]
\angles{\mathbf{X}(s,t)}{e_w} =: \int_{s}^{t} \ldots \int_{s}^{r_2} dX^{w_1}(r_1) \ldots dX^{w_k}(r_k)\;,
\end{equ}
for $w=w_1 \ldots w_k \in \mathcal{A}_p$. Of course, the integrals on the right hand side 
of \eqref{eq:iter_int} are not defined, as mentioned at the start of this section. 
Hence, for a given rough path $\mathbf{X}$, then the left hand side of 
\eqref{eq:iter_int} is the definition of the right hand side.

The conditions in Definition \ref{def:RP} ensure that the quantities \eqref{eq:iter_int} behave like iterated integrals. In particular, if $X$ is a smooth function and we define $\mathbf{X}$ by \eqref{eq:iter_int} in Young's sense, then $\mathbf{X}$ satisfies the conditions
of Definition~\ref{def:RP}, as was shown in \cite{Che54}. In particular, if $x = e_i$ and $y = e_j$, for any two letters $i, j \in \mathcal{A}$, then the first property gives
\begin{equ}
\angles{\mathbf{X}(s,t)}{e_i \otimes e_j} + \angles{\mathbf{X}(s,t)}{e_j \otimes e_i} = X^i(s,t) X^j(s,t)\;,
\end{equ}
where we write $X^i(s,t) := X^i(t) - X^i(s)$. This is the usual integration by parts formula. The second condition of Definition \ref{def:RP} provides the additivity property of the 
integral over consecutive intervals.

Given an $\alpha$-regular rough path $\mathbf{X}$, we define the following quantity
\begin{equ}[e:RPNorm]
\VERT \mathbf{X} \VERT_{\alpha} := \sum_{w \in \mathcal{A}_p \setminus \{\emptyset\}} \Vert \angles{\mathbf{X}}{e_w} \Vert_{\mathcal{B}^{\alpha |w|}}\;.
\end{equ}


\subsection{Controlled rough paths}

The theory of controlled rough paths was introduced in \cite{Gub04} for geometric rough paths of H\"{o}lder regularity from $\big(\frac{1}{3}, \frac{1}{2} \big]$. In \cite{Gub10}, 
the theory was generalised to rough paths of arbitrary positive regularity.

\begin{definition}
\label{def:RP_control}
Given $\alpha \in \bigl(0, \frac{1}{2}\bigr]$, $p = \lfloor 1/\alpha \rfloor$, a geometric rough path $\mathbf{X}$ of regularity $\alpha$, and a function $Y: \real \rightarrow \big(T^{(p-1)}\big(\real^n\big)\big)^*$ (the dual of the truncated tensor algebra), we say that $Y$ is {\it controlled} by $\mathbf{X}$ if, for every word $w \in \mathcal{A}_{p-1}$, one has the bound
\begin{equ}
|\angles{Y(t)}{e_w} - \angles{Y(s)}{\mathbf{X}(s,t) \otimes e_w}| \leq C |t-s|^{(p-|w|)\alpha}\;,
\end{equ}
for some constant $C > 0$.
\end{definition}

An alternative statement of Definition~\ref{def:RP_control} is that for every word $w \in \mathcal{A}_{p-1}$ there exists a function $R_Y^w \in \mathcal{B}^{(p-|w|)\alpha}$ such that
\begin{equ}[eq:Y_expansion]
\angles{Y(t)}{e_w} = \sum_{\bar{w} \in \mathcal{A}_{p-|w|-1}} \angles{Y(s)}{e_{\bar{w}} \otimes e_w} \angles{\mathbf{X}(s,t)}{e_{\bar{w}}} + R_Y^w(s,t)\;.
\end{equ}
Given an $\alpha$-regular geometric rough path $\mathbf{X}$, we then endow
the space of all controlled paths $Y$ with the semi-norm 
\[
\Vert Y \Vert_{\mathcal{C}^\alpha_{\mathbf{X}}} := \sum_{w \in \mathcal{A}_{p-1}} \Vert \angles{Y}{ e_w} \Vert_{\mathcal{C}^\alpha} + \sum_{w \in \mathcal{A}_{p-2}} \Vert R_Y^w \Vert_{\mathcal{B}^{(p-|w|)\alpha}}\;.
\]

Given a rough path $Y$ controlled by $\mathbf{X}$, one can define the integral \eqref{eq:integral_Y_X} by
\begin{equ}[eq:contr_int]
\dashint_{s}^t Y(r)\, dX^i(r) := \lim_{|\mathcal{P}| \rightarrow 0} \sum_{[u,v] \in \mathcal{P}} \Xi_{i}(u,v)\;,
\end{equ}
where we denoted $X^i(t) := \angles{\mathbf{X}(0,t)}{e_i}$ for $i \in \mathcal{A}$, and
\begin{equ}[e:defInt]
\Xi_{i}(u,v) := \sum_{w \in \mathcal{A}_{p-1}} \angles{Y(u)}{e_w} \angles{\mathbf{X}(u,v)}{e_w \otimes e_i}\;.
\end{equ}
Here, the limit is taken over a sequence of partitions $\mathcal{P}$ of the interval $[s,t]$, whose diameters $|\mathcal{P}|$ tend to $0$. It was proved in \cite[Theorem 8.5]{Gub10} that the rough integral \eqref{eq:contr_int} is well defined, i.e. the limit in \eqref{eq:contr_int} exists and is independent of the choice of partitions $\mathcal{P}$. 

If every coordinate $Y^j$ of the process $Y$ is controlled by $\mathbf{X}$, then we denote the rough integral of $Y$ with respect to $X$ by
\[
\left(\dashint_{s}^t Y(r) \otimes dX(r)\right)_{ij} := \dashint_{s}^t Y^j(r)\, dX^i(r)\;.
\]

We use the symbol $\dashint$ for the rough integral in \eqref{eq:contr_int}, in order to remind the abuse of notation, since the integral depends not only on $X^i$ and $Y^j$, but on much more information contained in $\mathbf{X}$ and $Y$.
In the following proposition we provide several bounds on the rough integrals.

\begin{proposition}
\label{prop:RP_integrals}
Let $Y$ be controlled by a geometric rough path $\mathbf{X}$ of regularity $\alpha \in \left(0, \frac{1}{2}\right]$. Then there is a constant $C$, independent of $Y$ and $\mathbf{X}$, such that 
\begin{equs}
\Big| \dashint_s^t Y(r) \otimes dX(r) - \Xi(s,t) \Big| &\leq C \VERT \mathbf{X} \VERT_{\alpha} \Vert Y \Vert_{\mathcal{C}^\alpha_{\mathbf{X}}} |t-s|^{\alpha (p+1)}\;, \label{eq:RI_approx_bound}\\
\Big\Vert \dashint_s^\cdot Y(r) \otimes dX(r) \Big\Vert_{\alpha} &\leq C \VERT \mathbf{X} \VERT_{\alpha} \Vert Y \Vert_{\mathcal{C}^\alpha_{\mathbf{X}}}\;.\label{eq:RI_bound}
\end{equs}

Moreover, if $\bar{Y}$ is controlled by another rough path $\mathbf{\bar{X}}$ of regularity $\alpha$, then there is a constant $C$, independent of $\mathbf{X}$, $\mathbf{\bar{X}}$, $Y$ and $\bar{Y}$, such that
\begin{equs}
\Big\Vert \dashint_s^\cdot Y(r) \otimes dX(r) - \dashint_s^\cdot \bar{Y}(r) &\otimes d \bar{X}(r) \Big\Vert_{\alpha} 
\leq C \VERT \mathbf{X} - \mathbf{\bar{X}} \VERT_{\alpha} \left(\Vert Y \Vert_{\mathcal{C}^\alpha_{\mathbf{X}}} + \Vert \bar{Y} \Vert_{\mathcal{C}^\alpha_{\mathbf{\bar{X}}}} \right) \\
&+ C \left( \VERT \mathbf{X} \VERT_{\alpha} + \VERT \mathbf{\bar{X}} \VERT_{\alpha} \right) \Vert Y, \bar{Y} \Vert_{\mathcal{C}^\alpha_{\mathbf{X}, \mathbf{\bar{X}}}},\label{eq:RI_two_bound}
\end{equs}
where we have used the quantity
\[
\Vert Y, \bar{Y} \Vert_{\mathcal{C}^\alpha_{\mathbf{X}, \mathbf{\bar{X}}}} := \sum_{w \in \mathcal{A}_{p-1}} \Vert \angles{Y}{e_w} - \angles{\bar{Y}}{ e_w} \Vert_{\mathcal{C}^\alpha} + \sum_{w \in \mathcal{A}_{p-2}} \Vert R_Y^w - R_{\bar{Y}}^w \Vert_{\mathcal{B}^{(p-|w|)\alpha}}\;.
\]
\end{proposition}
\begin{proof}
The bounds follow from \cite[Theorem 8.5, Proposition 6.1]{Gub10}.
\end{proof}

\begin{remark}
The notation $\VERT \mathbf{X} - \mathbf{\bar{X}} \VERT_{\alpha}$ is a slight abuse of
notation since $\mathbf{X} - \mathbf{\bar{X}}$ is not a rough path in general. The
definition \eqref{e:RPNorm} does however make perfect sense for the difference.
\end{remark}

In fact, the article \cite{Gub10} gives more precise bounds on the rough integrals than those provided in Proposition \ref{prop:RP_integrals}, but we prefer to have them in this form for the sake of conciseness.


\section{Definition and well-posedness of the solution}
\label{sec:notion}

Let us now give a short discussion of what we mean by ``solutions'' to
\eqref{eq:burgers}, as introduced in \cite{Hai11a}.
The idea is to find a process $X$ such that $v = u - X$ is of class $\mathcal{C}^1$ (in space), so that the definition of the integral \eqref{eq:test_nonlin} boils down to defining the integral
\begin{equ}
\int_{-\pi}^{\pi} \varphi(x) G\left(u(t,x)\right) d_x X(t,x)\;.
\end{equ}
If we have a canonical way of lifting $X$ to a rough path $\mathbf{X}$, 
this integral can be interpreted in the sense of rough paths.

A natural choice for $X$ is the solution to the linear stochastic heat equation. In order to get nice properties for this process, we build it in a slightly different way from \cite{Hai11a}. First, we define the stationary solution to the modified SPDE on the circle $\mathbb{T}$,
\begin{equ}[eq:Y]
d Y = \Delta Y dt + \Pi dW\;,
\end{equ}
where $\Pi$ denotes the orthogonal projection in $L^{2}$ onto the space of functions with zero mean. In particular, if we extend the cylindrical Brownian motion $W$ to whole $\mathbb{R}$ in time, then
\begin{equ}
 Y(t) = \int_{-\infty}^t S_{t-s}\, \Pi dW(s)\;,
\end{equ}
where $S$ is the heat semigroup defined below. Second, we define for all $(t,x) \in \mathbb{R}_+ \times \mathbb{T}$ the process
\begin{equ}[eq:X]
X(t,x) := Y(t,x) + \frac{1}{\sqrt{2\pi}} w^0(t)\;,
\end{equ}
where $w^0$ if the zeroth Fourier mode of $W$, i.e. $w^0$ is a Brownian motion.

\begin{remark}
We need to use $\Pi$ in \eqref{eq:Y} in order to obtain a stationary solution. In \cite{Hai11a}, the author used instead the stationary solution to $
d X = \Delta X dt - X dt + dW$ as a reference path. 
Our choice of $X$ was used in \cite{HMW14} and does not change the results of \cite{Hai11a}.
\end{remark}

The following lemma shows that there is a natural way to extend $X$ to a rough path.

\begin{lemma}\label{lem:lift}
For every $\frac{1}{3} < \alpha < \frac{1}{2}$, the stochastic process $X$ can be canonically lifted to a process $\mathbf{X}: \real \times \mathbb{T}^2 \rightarrow T^{(2)}\big(\real^n\big)$, such that for every fixed $t \in \real$, the process $\mathbf{X}(t)$ is a geometric $\alpha$-rough path.
\end{lemma}

The term ``canonically" means that for a large class of natural approximations of the process $X$ by smooth Gaussian processes $X_\varepsilon$, the iterated integrals of $X_\varepsilon$, defined by \eqref{eq:iter_int}, converge in $L^2$ to the corresponding elements of $\mathbf{X}$ (see \cite{FV10} for a precise definition and the proof).
Denote by $S_t = e^{t\Delta}$ the heat semigroup, which is given by convolution on the circle 
with the heat kernel
\begin{equ}[eq:heat_kernel]
p_{t}(x) = \frac{1}{\sqrt{2 \pi}} \sum_{k \in \intgr} e^{-t k^{2}} e^{ikx}\;.
\end{equ}
Assuming that the rough path-valued process $\mathbf{X}$ is given, we then \textit{define}
solutions to \eqref{eq:burgers} as follows:

\begin{definition}
\label{def:mild_burgers}
Setting $U(t) := S_t \left(u(0) - X(0) \right)$, a stochastic process $u$ is a {\it mild solution} to the equation \eqref{eq:burgers} if the process $v(t) := u(t) - X(t) - U(t)$ belongs to $\mathcal{C}^{1}_T$ for some $T > 0$ and the identity 
\begin{equs}[eq:mild_burgers]
v(t,x) &= \int_0^t S_{t-s} \left(F(u(s)) + G(u(s)) \partial_x (v(s) + U(s))\right)(x) \,ds \\
&\qquad + \int_0^t S_{t-s} \partial_x Z(s) (x) \,ds\;.
\end{equs}
holds for all $(t,x) \in [0, T] \times \mathbb{T}$. Here, we write for brevity $u(t) = v(t) + X(t) + U(t)$, and the process $Z(s,x)$ is a rough integral
\begin{equ}[eq:Z_def]
Z(s,x) := \dashint_{-\pi}^{x} G(u(s,y))\, d_y X(s,y)\;,
\end{equ}
whose derivative we consider in the sense of distributions. 
\end{definition}

\begin{remark}
In \cite{Hai11a}, the last integral in \eqref{eq:mild_burgers} was defined by
\begin{equ}
\int_0^t \dashint_{-\pi}^{\pi} p_{t-s}(x-y) G(u(s,y))\, d_y X(s,y) \,ds\;,
\end{equ}
but as noticed in \cite{HMW14}, the notion of solution in Definition~\ref{def:mild_burgers} is more convenient, as it simplifies treatment of the rough integral. This change does not affect the existence and uniqueness results of \cite{Hai11a}, and the resulting solutions are the same.
\end{remark}

For our convenience we rewrite the mild formulation of \eqref{eq:burgers_modif} as
\begin{equ}[eq:u_mild]
\bar{v}(t) = \mathbf{F}^{\bar{v}}(t) + \mathbf{G}^{\bar{v}}(t) + \mathbf{Z}^{\bar{v}}(t) - \mathbf{H}^{\bar{v}}(t)\;,
\end{equ}
where we have set
\begin{equs}[eq:u_mild_terms]
\mathbf{F}^{\bar{v}}(t) &:= \int_{0}^{t} S_{t-s} F(\bar{u}(s))\,ds\;, \qquad \mathbf{H}^{\bar{v}}(t)_i := \Lambda \int_{0}^{t} S_{t-s}~ \mathrm{div} G_{i}(\bar{u}(s))\,ds\;,\\
\mathbf{G}^{\bar{v}}(t) &:= \int_{0}^{t} S_{t-s} G(\bar{u}(s)) \partial_x ({\bar{v}}(s) + U(s)) \,ds\;,\\
\mathbf{Z}^{\bar{v}}(t) &:= \int_{0}^{t} S_{t-s} \partial_{x}Z(s)\,ds = \int_{0}^{t} \partial_{x}\big(S_{t-s} Z(s)\big)\,ds\;,
\end{equs}
and as before $\bar{u} = \bar{v} + X + U$, $U(t) = S_t(u^0 - X(0))$ and
\begin{equ}
Z(t,x) := \dashint_{-\pi}^{x} G(\bar{u}(t,y))\,d_{y} X(t,y)\;.
\end{equ}
Although the two terms $\mathbf{F}^{\bar{v}}$ and $\mathbf{H}^{\bar{v}}$ are of the same type,
we give them different names since they will arise in completely different ways
from the approximation.

\subsection{Existence and uniqueness results}

The next theorem provides the well-posedness result for a mild solution to the equation \eqref{eq:burgers}.
\begin{theorem}
Let us assume that $u^0 \in \mathcal{C}^\beta$ for some $\frac{1}{3} < \beta < \frac{1}{2}$. Furthermore, let $F \in \mathcal{C}^1$ and $G \in \mathcal{C}^{3}$. Then for almost every realisation of the driving noise, there is $T > 0$ such that there exists a unique mild solution to \eqref{eq:burgers} on the interval $[0,T]$ taking values in $\mathcal{C}\big([0,T], \mathcal{C}^\beta(\mathbb{T})\big)$. If moreover, $F$, $G$ and all their derivatives are bounded, then the solution is global (i.e. $T = \infty$).
\end{theorem}
\begin{proof} 
The proof can be done by performing a classical Picard iteration for $v$ given by \eqref{eq:mild_burgers} on the space $\mathcal{C}^{1}_{T}$ for some $T \leq 1$, see \cite{Hai11a}.
\end{proof}

\begin{remark}
\label{rem:v_reg}
The argument of \cite[Theorem 3.7]{Hai11a} also works in the space $\mathcal{C}^{1+\alpha}_{\alpha/2, T}$, for any $\alpha \in \bigl[0, \frac{1}{2}\bigr)$. Hence, the real regularity of $v(t)$ is $1 + \alpha$ rather than $1$. This fact will be used in Section \ref{sec:reaction} to estimate how close the approximate derivative of $v$ is to $\partial_x v$.
\end{remark}


\section{Solutions of the approximate equations}
\label{sec:approx}

In this section we rewrite the mild solution to the approximate equation \eqref{eq:SPDE_approx} in a way convenient for working in H\"{o}lder spaces of low regularity. In particular, we define the iterated integrals of higher order of the controlling process.

Similarly to \eqref{eq:Y} and \eqref{eq:X} we define the stationary process $Y_{\varepsilon}$ and $X_\varepsilon$ by
\begin{equs}[eq:X_eps]
d Y_{\varepsilon} = \Delta_{\varepsilon} Y_{\varepsilon} dt + \Pi H_{\varepsilon} d W\;, \qquad X_\varepsilon(t,x) := Y_\varepsilon(t,x) + \frac{1}{\sqrt{2\pi}} w^0(t)\;,
\end{equs}
where $w^0$ is the zeroth Fourier mode of $W$. 
Moreover, we define the approximate semigroup $S^{(\varepsilon)}_t = e^{t\Delta_{\varepsilon}}$ generated by the approximate Laplacian and given by convolution on the circle 
$\mathbb{T}$ with the approximate heat kernel
\begin{equ}[eq:approx_heat_kernel]
p^{(\varepsilon)}_{t}(x) = \frac{1}{\sqrt{2 \pi}} \sum_{k \in \intgr} e^{-t k^{2} m(\varepsilon k)} e^{ikx}\;.
\end{equ}
Furthermore, we define $U_\varepsilon(t) := S^{(\varepsilon)}_t \left(u_\varepsilon(0) - X_\varepsilon(0) \right)$ and $v_\varepsilon := u_\varepsilon - X_\varepsilon - U_\varepsilon$. Then the mild version of the approximate equation \eqref{eq:SPDE_approx} can be rewritten as
\begin{equs}[eq:u_eps_mild]
v_{\varepsilon}(t) = \mathbf{F}_{\varepsilon}^{v_{\varepsilon}}(t) + \mathbf{G}_{\varepsilon}^{v_{\varepsilon}}(t) + \int_{0}^{t} S^{(\eps)}_{t-s}G(u_{\varepsilon}(s))D_{\varepsilon} X_{\varepsilon}(s) \,ds\;,
\end{equs}
where we write for brevity $u_\varepsilon = v_\varepsilon + X_\varepsilon + U_\varepsilon$, and set
\begin{equs}[eq:Phi_Psi_eps]
\mathbf{F}_{\varepsilon}^{v_{\varepsilon}}(t) &:= \int_{0}^{t} S^{(\eps)}_{t-s}F(u_{\varepsilon}(s)) \,ds\;,\\
\mathbf{G}_{\varepsilon}^{v_{\varepsilon}}(t) &:= \int_{0}^{t} S^{(\eps)}_{t-s} G(u_\varepsilon(s)) D_\varepsilon \left(v_{\varepsilon}(s) + U_{\varepsilon}(s) \right) ds\;.
\end{equs}
As already mentioned in Section \ref{sec:RP}, the rough integrals are approximated by 
Riemann-like sums, but these include additional higher-order correction terms. Hence, we 
cannot expect in general that $Z(s,x)$, defined in \eqref{eq:Z_def}, is approximated by
\begin{equ}[eq:appr_Z]
\int_{-\pi}^{x} G(u_\varepsilon(s,y)) D_\varepsilon X_{\varepsilon}(s,y)\, dy\;,
\end{equ}
as $\varepsilon \downarrow 0$. In order to approximate $Z(s,x)$, we have to add some extra terms to \eqref{eq:appr_Z}. These extra terms give raise to the correction term in the limiting equation, mentioned in the introduction. In the rest of this section we build these missing extra terms.


\subsection{Iterated integrals}
\label{ssec:integrals}

In order to use the theory of rough paths with regularities close to zero, we need to build the iterated integrals of arbitrarily high orders of $X$ and $X_\varepsilon$ with respect to themselves. 

The expansion of $X_{\varepsilon}$ defined in \eqref{eq:X_eps} in the Fourier basis is given by
\begin{equs}[eq:X_eps_Fourier]
X_{\varepsilon}(t,x) =&\frac{1}{\sqrt{2 \pi}} w_0(t) + \frac{1}{\sqrt{2 \pi}} \sum_{k \in \intgr\backslash \{0\}} \int_{-\infty}^{t} e^{ikx} e^{-k^{2} m(\varepsilon k)(t-s)}h(\varepsilon k)\, dw_{k}(s) \\
=& \frac{1}{\sqrt{2 \pi}} w_0(t) + \frac{1}{\sqrt{\pi}} \sum_{k =1}^\infty \frac{q^{(\varepsilon)}_{k}}{k} \left(\eta^{(\varepsilon)}_{k} (t) \sin(kx) + \eta^{(\varepsilon)}_{-k} (t) \cos(kx)\right).
\end{equs}
Here, $w_{k}$ are $\com^{n}$-valued standard Brownian motions (i.e. the real and imaginary parts of every component are independent real-valued Brownian motions so that $\mathbb{E}|w^{i}_{k}(t)|^{2}=t$), which are independent up to the constraint $w_{k}=\bar{w}_{-k}$ ensuring that $X_{\varepsilon}$ is real-valued. Furthermore, for every fixed $t \geq 0$, $\eta^{(\varepsilon)}_{k}(t)$ are independent $\real^{n}$-valued standard Gaussian random vectors such that
\begin{equ}
\mathbb{E}\left[ \eta^{(\varepsilon)}_{k}(0) \otimes \eta^{(\varepsilon)}_{k}(t) \right] = e^{-k^2 m(\varepsilon k) t} \mathrm{Id}\;,
\end{equ}
and the coefficients $q^{(\varepsilon)}_{k}$ are defined by
\begin{equ}[eq:q_def]
q^{(\varepsilon)}_{k} = \frac{h(\varepsilon k)}{\sqrt{m(\varepsilon k)}} \qquad\mbox{ for  $k \geq 1$}\;.
\end{equ}
Similarly, the Fourier expansion of the process $X$ is
\begin{equs}[eq:X_Fourier]
X(t,x) = \frac{1}{\sqrt{2 \pi}} w_0(t) + \frac{1}{\sqrt{\pi}} \sum_{k =1}^\infty \frac{1}{k} \left(\eta_{k} (t) \sin(kx) + \eta_{-k} (t) \cos(kx)\right),
\end{equs}
where $\eta_{k}(t)$ are independent $\real^{n}$-valued standard Gaussian random vectors such that
\begin{equ}
\mathbb{E}\left[ \eta_{k}(0) \otimes \eta_{k}(t) \right] = e^{-k^2 t} \mathrm{Id}\;.
\end{equ}
Furthermore, the random vectors $\{(\eta^{(\varepsilon)}_{k}(t), \eta_{k}(t)) : k \in \intgr \setminus \{0\}\}$ are independent and satisfy
\begin{equ}
\mathbb{E}\left[ \eta^{(\varepsilon)}_{k}(t) \otimes \eta_{k}(t) \right] = \frac{\sqrt{m(\varepsilon k)}}{m(\varepsilon k) + 1} \mathrm{Id} =: \tilde{q}^{(\varepsilon)}_k\;.
\end{equ}

The following lemma provides bounds on the canonical lifts of $X(t)$ and $X_\varepsilon(t)$ to Gaussian rough paths.

\begin{lemma}
\label{lem:X_X_eps}
For any $\alpha \in \bigl(0,\frac{1}{2}\bigr)$ and $p = \lfloor 1/\alpha \rfloor$, 
there are canonical lifts \linebreak $\mathbf{X}, \mathbf{X}_\eps: \mathbb{R}_+ \times \mathbb{T}^2 \to T^{(p)}(\mathbb{R}^n)$ 
of the processes $X$ and $X_\eps$ respectively, which are continuous functions in the time variable such that, for every $t \geq 0$, $\mathbf{X}(t)$ and $\mathbf{X}_\eps(t)$ are Gaussian rough paths of regularity $\alpha$. Furthermore, for any $\lambda < \frac{1}{2}-\alpha$ and any $T > 0$ the following bounds hold
\begin{equ}[eq:X_bounds]
\mathbb{E}\Vert X \Vert_{\mathcal{C}^\alpha_T} \lesssim 1\;,\qquad 
\mathbb{E}\Vert X - X_\varepsilon \Vert_{\mathcal{C}^\alpha_T} \lesssim \varepsilon^{\lambda}\;.
\end{equ}
Moreover, for any word $w \in \mathcal{A}_p$ with $|w| \geq 2$ we have
\begin{equ}[eq:X_iterated_bounds]
\mathbb{E}\Vert \mathbf{X}^w \Vert_{\mathcal{B}^{|w|\alpha}_T} \lesssim 1\;,\qquad
\mathbb{E}\Vert \mathbf{X}^w - \mathbf{X}^w_\varepsilon \Vert_{\mathcal{B}^{|w|\alpha}_T} \lesssim \varepsilon^{\lambda}\;,
\end{equ}
where we use the notation $\mathbf{X}^w = \angles{\mathbf{X}}{e_w}$.
\end{lemma}
\begin{proof}
The proof of \eqref{eq:X_bounds} is provided in \cite[Lemma 3.3]{HMW14}. We only have to show that there exist the claimed lifts which satisfy the estimates \eqref{eq:X_iterated_bounds}. To this end, we define, for some $\kappa >0$, the following sequences
\[
\beta_k^{(\varepsilon, \kappa)} = \frac{h(\varepsilon k)^2}{k^{\kappa} m(\varepsilon k)}\;,\qquad \rho_k^{(\varepsilon, \kappa)} = \frac{h(\varepsilon k)}{k^{\kappa}(m(\varepsilon k) + 1)}\;,
\]
where $k \geq 1$. First, for the increments of $\beta_k^{(\varepsilon, \kappa)}$ we have
\begin{equs}
|\beta_{k+1}^{(\varepsilon, \kappa)} - \beta_k^{(\varepsilon, \kappa)}| \leq |(q^{(\varepsilon)}_{k+1})^2| &\left|(k+1)^{-\kappa} - k^{-\kappa}\right| \\
&+ k^{-\kappa} | (q^{(\varepsilon)}_{k+1})^2 - (q^{(\varepsilon)}_{k})^2 | \leq C k^{-1 - \kappa}\;,
\end{equs}
for some constant $C > 0$, where $q^{(\varepsilon)}_{k}$ is defined in \eqref{eq:q_def}. To get the last inequality we have used the bounds on the functions $m$ and $h$, provided in Assumptions \ref{as:second} and \ref{as:fourth}, and the estimate
\[
| (q^{(\varepsilon)}_{k+1})^2 - (q^{(\varepsilon)}_{k})^2 | \leq C k^{-1}\;,
\]
which follows from the bound on the total variation of the function $h^2/m$, provided by Assumption \ref{as:fourth}. Second, the convergence $\beta_k^{(\varepsilon, \kappa)} \log k \rightarrow 0$ holds as $k \rightarrow \infty$. 

Using these properties of $\beta_k^{(\varepsilon, \kappa)}$, we obtain from \cite[Theorem 4]{Tel73} that the series $\sum_{k=1}^N \beta_k^{(\varepsilon, \kappa)} \cos kx$ converge in $L^1$ as $N \rightarrow \infty$, and the $L^1$-norm of the limit is independent of $\varepsilon$, which proves that for any $\kappa > 0$ the parametrized sequence $\beta_k^{(\varepsilon, \kappa)}$ is uniformly negligible in $\varepsilon \in (0,1)$ in the sense of \cite[Definition 3.6]{FGGR13}.

Similarly, using the bound on the total variation of $h/(m+1)$, which is stated in Assumption $\ref{as:fourth}$, we can obtain that for any $\kappa > 0$ the sequence $\rho_k^{(\varepsilon, \kappa)}$ is uniformly negligible in $\varepsilon \in (0,1)$ as well.

Noticing that the coefficients of the Fourier expansions \eqref{eq:X_eps_Fourier} and \eqref{eq:X_Fourier} satisfy
\[
\left(\frac{q^{(\varepsilon)}_{k}}{k}\right)^2 = \frac{\beta_k^{(\varepsilon, \kappa)}}{k^{2 - \kappa}}\;, \qquad\frac{q^{(\varepsilon)}_{k} \tilde{q}^{(\varepsilon)}_{k}}{k^2} = \frac{\rho_k^{(\varepsilon, \kappa)}}{k^{2 - \kappa}}\;,
\]
we can apply \cite[Theorem 3.16]{FGGR13} and obtain that for every $t$ and $\alpha < \frac{1}{2}$ the processes $X(t)$ and $X_\varepsilon(t)$ can indeed be lifted to $\alpha$-regular rough paths $\mathbf{X}(t)$ and $\mathbf{X}_\varepsilon(t)$ respectively, such that for any $q \geq 1$ and for any word $w \in \mathcal{A}_p$ with $|w| \geq 2$ the bounds
\begin{equs}[eq:X_iterated_space_bound]
\mathbb{E}\Vert \mathbf{X}^w(t) \Vert^q_{\mathcal{B}^{|w|\alpha}} \lesssim 1\;, \qquad \mathbb{E}\Vert \mathbf{X}^w_\varepsilon(t) \Vert^q_{\mathcal{B}^{|w|\alpha}} \lesssim 1
\end{equs}
hold uniformly in $t \in [0,T]$. Furthermore, by \cite[Theorem 3.15]{FGGR13} we obtain that for all $\gamma < \frac{1}{2} - \alpha$, any $q \geq 1$ and any $\kappa > 0$ small enough,
\begin{equ}[eq:X_X_eps_iterated_space_bound]
\mathbb{E}\Vert \mathbf{X}^w(t) - \mathbf{X}^w_\varepsilon(t) \Vert^q_{\mathcal{B}^{|w|\alpha}} \lesssim \left( \sup_{x \in \mathbb{T}} \mathbb{E} |X(t,x) - X_\varepsilon(t,x)|^2 \right)^{(\gamma+\kappa)q} \lesssim \varepsilon^{\gamma q}\;,
\end{equ}
uniformly in $t \in [0,T]$. The last bound can be shown almost identically to \cite[(3.16d)]{HMW14}, but taking $\theta \equiv 1$ and the time interval from $-\infty$.

Now we will investigate the temporal regularity of $\boldsymbol{X}_\varepsilon$. Our aim is to apply \cite[Theorem 3.15]{FGGR13} to the processes $\boldsymbol{X}_\varepsilon(s)$ and $\boldsymbol{X}_\varepsilon(t)$, with $s, t \in [0,T]$. To this end, let us define $\tau = |t-s|$ and the parametrized sequence $\mu^{(\tau, \varepsilon)}_k = e^{-k^2 m(\varepsilon k) \tau}$. Then, in the same way as in the beginning of the proof and using Assumptions \ref{as:second} and \ref{as:fourth}, we obtain that for any $\kappa > 0$ the sequence $\beta_k^{(\kappa, \varepsilon)} \mu^{(\tau, \varepsilon)}_k$ is uniformly negligible in $\tau > 0$ and $\varepsilon \in (0,1)$ and by \cite[Theorem 3.15]{FGGR13} we obtain, for any word $w \in \mathcal{A}_p$ with $|w| \geq 2$,
\begin{equ}[eq:X_eps_iterated_temp_bound]
\mathbb{E}\Vert \mathbf{X}^w_\varepsilon(t) - \mathbf{X}^w_\varepsilon(s) \Vert^q_{\mathcal{B}^{|w|\alpha}} \lesssim \left( \sup_{x \in \mathbb{T}} \mathbb{E} |X_\varepsilon(s,x) - X_\varepsilon(t,x)|^2 \right)^{\gamma q} \lesssim |t-s|^{\frac{\gamma q}{2}}\;,
\end{equ}
for all $\gamma < \frac{1}{2} - \alpha$ and $q \geq 1$. Here, the last bound can be derived similarly to \cite[(3.16a)]{HMW14}, but with $\theta \equiv 1$ and the time interval from $-\infty$. In the same way, we get 
\begin{equs}[eq:X_iterated_temp_bound]
\mathbb{E}\Vert \mathbf{X}^w(t) - \mathbf{X}^w(s) \Vert^q_{\mathcal{B}^{|w|\alpha}} \lesssim |t-s|^{\frac{\gamma q}{2}}\;.
\end{equs}
Applying the Kolmogorov criterion \cite{Kal02} together with the bounds \eqref{eq:X_iterated_space_bound} and \eqref{eq:X_iterated_temp_bound}, we get the first estimate in \eqref{eq:X_iterated_bounds}.

Now, let us take any word $w \in \mathcal{A}_p$ with $|w| \geq 2$. Then, on the one hand, the estimate  \eqref{eq:X_X_eps_iterated_space_bound} gives for every $q \geq 1$,
\begin{equs}
\mathbb{E} \Vert \boldsymbol{X}^w(t) &- \boldsymbol{X}_\varepsilon^w(t) - \boldsymbol{X}^w(s) + \boldsymbol{X}_\varepsilon^w(s) \Vert^q_{\mathcal{B}^{\alpha|w|}} \\
&\leq \mathbb{E}\Vert \boldsymbol{X}^w(t) - \boldsymbol{X}^w_\varepsilon(t) \Vert^q_{\mathcal{B}^{\alpha|w|}} + \mathbb{E}\Vert \boldsymbol{X}^w(s) - \boldsymbol{X}^w_\varepsilon(s) \Vert^q_{\mathcal{B}^{\alpha|w|}} \lesssim \varepsilon^{\gamma q}\;.
\end{equs}
On the other hand, from \eqref{eq:X_iterated_temp_bound} and \eqref{eq:X_eps_iterated_temp_bound} the following estimate follows
\begin{equs}
\mathbb{E}&\Vert \boldsymbol{X}^w(t) - \boldsymbol{X}_\varepsilon^w(t) - \boldsymbol{X}^w(s) + \boldsymbol{X}_\varepsilon^w(s) \Vert^q_{\mathcal{B}^{\alpha|w|}} \\
&\leq \mathbb{E}\Vert \boldsymbol{X}^w_\varepsilon(t) - \boldsymbol{X}^w_\varepsilon(s) \Vert^q_{\mathcal{B}^{\alpha|w|}} + \mathbb{E}\Vert \boldsymbol{X}^w(t) - \boldsymbol{X}^w(s) \Vert^q_{\mathcal{B}^{\alpha|w|}} \lesssim |t-s|^{\frac{\gamma q}{2}}\;.
\end{equs}
Combining these two bunds we obtain
\begin{equs}
\mathbb{E}\Vert \boldsymbol{X}^w(t) - \boldsymbol{X}_\varepsilon^w(t) - \boldsymbol{X}^w(s) + \boldsymbol{X}_\varepsilon^w(s) \Vert^q_{\mathcal{B}^{\beta|w|}} &\lesssim \left(\varepsilon^{\gamma} \wedge |t-s|^{\frac{\gamma}{2}}\right)^q \\
&\lesssim \left(\varepsilon^{\frac{1}{2} - \alpha - \delta} |t-s|^{\frac{\delta}{2}}\right)^q,
\end{equs}
for any $\delta > 0$ small enough and uniformly in $s, t \in [0,T]$. From this bound, estimate \eqref{eq:X_X_eps_iterated_space_bound} and the Kolmogorov criterion \cite{Kal02} we obtain the second bound in \eqref{eq:X_iterated_bounds}.
\end{proof}


\subsection{Approximation of the rough integral}
\label{ssec:rough_int}

Now, having defined the iterated integrals of $X_\varepsilon$, we can build an approximation of the process $Z$ defined in \eqref{eq:Z_def}.

The idea comes from the fact that if $u(t)$ is controlled by $\mathbf{X}(t)$, then the process $G(u(t))$ is controlled by $\mathbf{X}(t)$ as well. The Taylor expansion gives an approximation for $G_{ij}(u(t))$,
\begin{equs}
G_{ij}(u(t, y)) \approx G_{ij}(u(t, x)) + \sum_{w \in \mathcal{A}_{p-1} \setminus \emptyset} \tilde C_w D^w G_{ij}(u(t, x)) \left( u(t, y) - u(t, x) \right)_w.
\end{equs}
Here, $\tilde C_w$ are combinatorial factors which can be calculated explicitly. Furthermore, we use the following notation: for $w = w_1 \cdots w_k \in \mathcal{A}_{p-1}$ and $k \geq 1$ we denote $D^w = D^{w_1} \cdots D^{w_k}$ and $u(t,x)_w = u_{w_1}(t,x) \cdots u_{w_k}(t,x)$.

Recalling that we will look for solutions such that $u(t) - X(t) \in \mathcal{C}^1$, we obtain an approximation of $G_{ij}(u(t))$ via $\mathbf{X}(t)$,
\begin{equ}
G_{ij}(u(t, y)) \approx G_{ij}(u(t, x)) + \sum_{\substack{w \in \mathcal{A}_{p-1} \setminus \{\emptyset\} \\ w=w_1 \ldots w_k}} \tilde C_w D^w G_{ij}(u(t, x)) \prod_{l=1}^k \angles{\mathbf{X}(t; x,y)}{e_{w_l}}.
\end{equ}
Symmetrising this expression and using Definition~\ref{def:RP}, this can be rewritten as
\begin{equ}[eq:G_approx_new]
G_{ij}(u(t, y)) \approx \sum_{w \in \mathcal{A}_{p-1}} C_w D^w G_{ij}(u(t, x)) \angles{\mathbf{X}(t; x,y)}{e_w}\;,
\end{equ}
for some slightly different constants $C_w$. This expansion motivates our choice of the terms in the approximation of the rough integral.

In view of Assumption \ref{as:third}, it is natural to define the process $D_\varepsilon \mathbf{X}_{\varepsilon} : \real_+ \times \mathbb{T} \rightarrow T^{(p)}\big(\real^n\big)$ in the following way: for any word $w \in \mathcal{A}_p$ we set
\begin{equ}[eq:D_eps_X_def]
\angles{D_\varepsilon \mathbf{X}_{\varepsilon}(t;y)}{e_w} := \frac{1}{\varepsilon} \int_{\real} \angles{\mathbf{X}_\varepsilon(t;y,y+\varepsilon z)}{e_w} \mu(dz)\;.
\end{equ}
Combining the expansion \eqref{eq:G_approx_new} with the definition \eqref{e:defInt}, 
it appears plausible that a good approximation of $Z$ is given by
\begin{equ}[eq:Z_eps_def]
Z_{\varepsilon}(t,x)_{i} := \sum_{w \in \mathcal{A}_{p-1}} C_w \int_{-\pi}^{x} D^{w} G_{ij}(u_{\varepsilon}(t,y)) \angles{D_\varepsilon \mathbf{X}_{\varepsilon}(t;y)}{e_w \otimes e_j} \,dy\;.
\end{equ}
Here, to simplify the notation we have omitted the sum over $j$. 

Now we can rewrite the mild solution \eqref{eq:u_eps_mild} as
\begin{equ}[eq:u_eps_mild_new]
v_{\varepsilon}(t) = \mathbf{F}_{\varepsilon}^{v_{\varepsilon}}(t) + \mathbf{G}_{\varepsilon}^{v_{\varepsilon}}(t) + \mathbf{Z}_{\varepsilon}^{v_{\varepsilon}}(t) - \mathbf{H}_{\varepsilon}^{v_{\varepsilon}}(t) - \mathbf{\bar H}_{\varepsilon}^{v_{\varepsilon}}(t)\;,
\end{equ}
where the functions $\mathbf{F}_{\varepsilon}^{v_{\varepsilon}}$ and $\mathbf{G}_{\varepsilon}^{v_{\varepsilon}}$ are defined in \eqref{eq:Phi_Psi_eps}. The term involving the rough integral is denoted by
\begin{equs}[eq:Xi_eps_def]
\mathbf{Z}_{\varepsilon}^{v_{\varepsilon}}(t) := \int_{0}^{t} S^{(\eps)}_{t-s}\partial_{x}Z_{\varepsilon}(s)\,ds = \int_{0}^{t} \partial_{x} \big(S^{(\eps)}_{t-s}Z_{\varepsilon}(s)\big) \,ds\;.
\end{equs}
The additional terms in \eqref{eq:u_eps_mild_new} which we used to approximate the rough integral we denote by
\begin{equs}
\label{eq:Upsilon_eps}
\mathbf{H}_{\varepsilon}^{v_{\varepsilon}}(t,x)_{i} &:= \sum_{k \in \mathcal{A}} \int_{0}^{t} S^{(\eps)}_{t-s}\Big(D^{k} G_{ij}(u_{\varepsilon}(s,\cdot)) \angles{D_\varepsilon \mathbf{X}_{\varepsilon}(s;\cdot)}{e_{kj}}\Big)(x)\, ds\;,\\
\mathbf{\bar H}_{\varepsilon}^{v_{\varepsilon}}(t,x)_{i} &:= \sum_{\substack{w \in \mathcal{A}_{p-1}\\ |w| \geq 2}} C_w \int_{0}^{t} S^{(\varepsilon)}_{t-s}\Big(D^{w} G_{ij}(u_{\varepsilon}(s,\cdot)) \angles{D_\varepsilon \mathbf{X}_{\varepsilon}(s;\cdot)}{e_{wj}}\Big)(x)\, ds.
\end{equs}

In the next sections we will show that the term $\mathbf{\bar H}_{\varepsilon}^{v_{\varepsilon}}$ tends to $0$ and the other terms in \eqref{eq:u_eps_mild_new} converge to the corresponding terms in \eqref{eq:u_mild} in the space $\mathcal{C}^1_T$.

\subsection{A priori estimates on the terms}
\label{ss:APriori}

In what follows we use the constant $\alpha_{\star} = \frac{1}{2} - \alpha$, for some fixed small $\alpha > 0$. This constant represents the real spatial regularity of the process $X$ defined in \eqref{eq:X}. To obtain better bounds we will work in the spaces of regularity $\alpha$, which is close to $0$. The constants $\alpha$ and $\alpha_\star$ are used throughout the article as fixed values. 

To shorten notations we define the norm
\begin{equ}[eq:RPNorm_time]
\VERT \mathbf{X} \VERT_{\alpha_\star, T} := \sup_{t \in [0,T]} \VERT \mathbf{X}(t) \VERT_{\alpha_\star}\;.
\end{equ}
See \eqref{e:RPNorm} for the definition of the norm of a rough path. For any $K > 0$ we define the stopping time
\begin{equs}
\sigma_{K} := \inf\{t \geq 0\,:\, &\Vert X \Vert_{\mathcal{C}_{t}^{\alpha_\star}} \geq K, \text{ or } \VERT \mathbf{X} \VERT_{\alpha_\star, t} \geq K, \text{ or } \Vert \bar{v} \Vert_{\mathcal{C}^{1+\alpha_\star}_{\alpha_\star/2,t}} \geq K,\\
&\text{ or } \Vert \bar{v} \Vert_{\mathcal{C}^{1}_{t}} \geq K, \text{ or } \Vert v_\varepsilon \Vert_{\mathcal{C}^{1}_{t}} \geq K\}\;.
\end{equs}
Note that in view of Remark \ref{rem:v_reg}, the condition on the norm $\Vert \bar{v} \Vert_{\mathcal{C}^{1+\alpha_\star}_{\alpha_\star/2, t}}$ is reasonable. For any two letters $i, j \in \mathcal{A}$ we define the process
\begin{equ}
\mathcal{H}^{i, j}_\varepsilon(t,x) := \Lambda \delta_{i, j} - \angles{D_\varepsilon \mathbf{X}_{\varepsilon}(t;x)}{e_i \otimes e_j}\;,
\end{equ}
where $\delta$ is the Kronecker delta. To have a priori bounds on the corresponding $\varepsilon$-quantities we introduce the stopping time
\begin{equs}
&\sigma_{K, \varepsilon} := \inf\{t \geq 0\,:\, \Vert X - X_\varepsilon \Vert_{\mathcal{C}_{t}^{\alpha_\star}} \geq 1, ~\text{ or }~ \VERT \mathbf{X} - \mathbf{X}_\varepsilon \VERT_{\alpha_\star, t} \geq 1,\\
&~\text{ or }~ \Vert \mathcal{H}_\varepsilon \Vert_{\mathcal{C}_{t}^{-\frac{1}{2} + \alpha}} \geq 1, ~\text{ or }~ \Vert \bar{v} - v_\varepsilon \Vert_{\mathcal{C}^{\alpha}_{t}} \geq 1, ~\text{ or }~ \Vert \bar{v} - v_\varepsilon \Vert_{\mathcal{C}^{1}_{(1-\alpha)/2,t}} \geq 1 \}\;.
\end{equs}
The blow-up of the norm $\Vert \bar{v}(t) - v_\varepsilon(t) \Vert_{\mathcal{C}^{1}}$ comes from the regularization property of the heat semigroup and the fact that we work in the $\alpha$-regular spaces, i.e. we use the bound
\[
\Vert U(t) \Vert_{\mathcal{C}^1} \lesssim t^{\frac{\alpha-1}{2}} \left( \Vert u^0 \Vert_{\mathcal{C}^\alpha} + \Vert X(0) \Vert_{\mathcal{C}^\alpha} \right)\;.
\]
See Appendix \ref{appx:heat_sg} for the properties of the heat semigroup. Finally, for $T > 0$ whose value will be chosen in the proof of Theorem~\ref{thm:first}, we define the stopping time $\varrho_{K,\varepsilon} := \sigma_{K} \wedge \sigma_{K, \varepsilon} \wedge T$ and write in what follows 
\begin{equ}[e:StoppedTime]
t_{\varepsilon} := t \wedge \varrho_{K, \varepsilon}\;.
\end{equ}

\begin{remark}
\label{rem:stop_time}
In the article we always consider time intervals up to the stopping time $\varrho_{K, \varepsilon}$. Therefore, all the quantities involved in the definition of $\varrho_{K, \varepsilon}$ are bounded by $K + 1$ and all the proportionality constants can depend on $K$.
\end{remark}

Before providing a proof of Theorem~\ref{thm:first}, we establish in the following three sections certain bounds on the terms of \eqref{eq:u_mild} and \eqref{eq:u_eps_mild_new}.


\section{Estimates on the reaction term}
\label{sec:reaction}

In this section we prove convergence of the reaction terms of the approximate equation \eqref{eq:u_eps_mild_new} to the corresponding terms of \eqref{eq:u_mild}. Let us recall the notation \eqref{e:StoppedTime} and Remark \ref{rem:stop_time}, which says that all the quantities involved in the definition of the stopping time $\varrho_{K,\varepsilon}$ are bounded on the interval $(0, t_\varepsilon]$ by the constant $K + 1$ and all the proportionality constants below can depend on $K$.

The next proposition gives a bound on the terms $\mathbf{G}^{\bar{v}}$ and $\mathbf{G}_{\varepsilon}^{v_{\varepsilon}}$ defined in \eqref{eq:u_mild_terms} and \eqref{eq:Phi_Psi_eps} respectively.

\begin{proposition}
\label{prop:Psi_Psi_eps}
For any $\gamma \in (0, 1]$, $t > 0$ and $\kappa > 0$ small enough the following bound holds
\begin{equs}[eq:Psi_Psi_eps]
\Vert \mathbf{G}^{\bar{v}}(t_\varepsilon) - \mathbf{G}_{\varepsilon}^{v_{\varepsilon}}(t_\varepsilon) \Vert_{\mathcal{C}^{\gamma}} &\lesssim t_\varepsilon^{\frac{1 + \alpha - \gamma}{2}} \left( \Vert \bar{v} - v_{\varepsilon} \Vert_{\mathcal{C}^{\alpha}_{t_\varepsilon}} + \Vert \bar{v} - v_{\varepsilon} \Vert_{\mathcal{C}^{1}_{(1-\alpha)/2, t_\varepsilon}} \right) \\
&\quad+\Vert X - X_{\varepsilon} \Vert_{\mathcal{C}^{\alpha}_{t_\varepsilon}} + \Vert u^0 - u^0_{\varepsilon} \Vert_{\mathcal{C}^{\alpha}} + \varepsilon^{\alpha_\star - \kappa}\;.
\end{equs}
\end{proposition}
\begin{proof}
For any $t > 0$, using the notation \eqref{e:StoppedTime}, we can rewrite
\begin{equs}
\mathbf{G}^{\bar{v}}(t_\varepsilon) &- \mathbf{G}_{\varepsilon}^{v_{\varepsilon}}(t_\varepsilon) = \int_{0}^{t_\varepsilon} S_{t_\varepsilon-s} G(\bar{u}(s)) \big(\partial_x \bar{v}(s) - D_\varepsilon \bar{v}(s) \big) \,ds \\
&+ \int_{0}^{t_\varepsilon} S_{t_\varepsilon-s} G(\bar{u}(s)) \big(\partial_x U(s) - D_\varepsilon U(s) \big) \,ds \\
&+ \int_{0}^{t_\varepsilon} S_{t_\varepsilon-s} G(\bar{u}(s)) \big(D_\varepsilon \bar{v}(s) - D_\varepsilon v_{\varepsilon}(s) \big) \,ds \\
&+ \int_{0}^{t_\varepsilon} S_{t_\varepsilon-s} G(\bar{u}(s)) \big(D_\varepsilon U(s) - D_\varepsilon U_{\varepsilon}(s) \big) \,ds \\
&+ \int_{0}^{t_\varepsilon} S_{t_\varepsilon-s} \big( G(\bar{u}(s)) - G(u_\varepsilon(s))\big)  D_\varepsilon \left(v_{\varepsilon}(s) + U_{\varepsilon}(s) \right) \,ds \\
&+ \int_{0}^{t_\varepsilon} \big( S_{t_\varepsilon-s} - S^{(\eps)}_{t_\varepsilon-s} \big) G(u_\varepsilon(s)) D_\varepsilon \left(v_{\varepsilon}(s) + U_{\varepsilon}(s) \right) \,ds =: \sum_{1\leq j \leq 6} J_j\;.
\end{equs}

To bound the term $J_1$, we first investigate how good the operator $D_\varepsilon$ approximates $\partial_x$. Let us take a function $\varphi \in \mathcal{C}^{1+\alpha_\star}(\mathbb{T})$. Then by the Assumption \ref{as:third}, we can rewrite
\begin{equs}
\left(D_\varepsilon - \partial_x \right) \varphi(x) = \frac{1}{\varepsilon} \int_{\real} \left(\varphi(x + \varepsilon y) - \varphi(x) - \partial_x \varphi(x) \varepsilon y \right) \mu(dy)\;.
\end{equs}
Using the fact, that the H\"{o}lder regularity of $\varphi$ is $1 + \alpha_\star$, we obtain
\[
\left|\varphi(x + \varepsilon y) - \varphi(x) - \partial_x \varphi(x) \varepsilon y \right| \lesssim |\varepsilon y|^{1 + \alpha_\star} \Vert \varphi \Vert_{\mathcal{C}^{1+\alpha_\star}}\;.
\]
This yields the estimate
\begin{equs}[eq:D_approx]
\Vert \left(D_\varepsilon - \partial_x \right) \varphi \Vert_{\mathcal{C}^0} \lesssim \varepsilon^{\alpha_\star} \Vert \varphi \Vert_{\mathcal{C}^{1+\alpha_\star}}\;,
\end{equs}
where we have used the boundedness of the $(1+\alpha_\star)$th moment of $\mu$. 

Using this estimate we derive
\begin{equs}
\Vert J_1 \Vert_{\mathcal{C}^{\gamma}} &\leq \int_{0}^{t_\varepsilon} \Vert S_{t_\varepsilon-s} \Vert_{\mathcal{C}^{0} \rightarrow \mathcal{C}^{\gamma}} \Vert G(\bar{u}(s)) \Vert_{\mathcal{C}^{0}} \Vert \partial_x \bar{v}(s) - D_\varepsilon \bar{v}(s) \Vert_{\mathcal{C}^{0}} \,ds\\
&\lesssim \varepsilon^{\alpha_\star} \int_{0}^{t_\varepsilon} (t_\varepsilon-s)^{-\frac{\gamma}{2}} \Vert \bar{v}(s) \Vert_{\mathcal{C}^{1 + \alpha_\star}} \,ds \lesssim \varepsilon^{\alpha_\star} t_\varepsilon^{1 - \frac{\gamma + \alpha_\star}{2}}\;,\label{eq:Psi_1}
\end{equs}
where we have used boundedness of $\Vert \bar{u} \Vert_{\mathcal{C}^{0}_{t_\varepsilon}}$ and $\Vert \bar{v} \Vert_{\mathcal{C}^{1+\alpha_\star}_{\alpha_\star/2, t_\varepsilon}}$.

To derive a bound on $J_2$, we notice that
\[
\Vert U(s) \Vert_{\mathcal{C}^{1 + \alpha_\star}} \lesssim s^{-\frac{1}{2}} \left(\Vert u^0 \Vert_{\mathcal{C}^{\alpha_\star}} + \Vert X(0) \Vert_{\mathcal{C}^{\alpha_\star}} \right),
\]
which follows from Lemma \ref{lem:S_reg}. Hence, using the estimate \eqref{eq:D_approx} for $U$, we obtain
\begin{equs}
\Vert J_2 \Vert_{\mathcal{C}^{\gamma}} &\leq \int_{0}^{t_\varepsilon} \Vert S_{t_\varepsilon-s} \Vert_{\mathcal{C}^{0} \rightarrow \mathcal{C}^{\gamma}} \Vert G(\bar{u}(s)) \Vert_{\mathcal{C}^{0}} \Vert \partial_x U(s) - D_\varepsilon U(s) \Vert_{\mathcal{C}^{0}} \,ds\\
&\lesssim \varepsilon^{\alpha_\star} \int_{0}^{t_\varepsilon} (t_\varepsilon-s)^{-\frac{\gamma}{2}} \Vert U(s) \Vert_{\mathcal{C}^{1 + \alpha_\star}} \,ds \lesssim \varepsilon^{\alpha_\star} t_\varepsilon^{\frac{1 - \gamma}{2}}\;.\label{eq:Psi_2}
\end{equs}

Note, that for any function $\varphi \in \mathcal{C}^1(\mathbb{T})$ we have by Assumption \ref{as:third},
\begin{equ}[eq:D_eps_bound]
|D_\varepsilon \varphi(x)| \leq \frac{1}{\varepsilon} \int_{\real} \int_0^{\varepsilon |z|} |\partial_x \varphi(x + y)| dy |\mu|(dz) \lesssim \Vert \varphi \Vert_{\mathcal{C}^{1}}\;.
\end{equ}
Using this bound we obtain 
\begin{equs}
\Vert &J_3 \Vert_{\mathcal{C}^{\gamma}} \leq \int_{0}^{t_\varepsilon} \Vert S_{t_\varepsilon-s} \Vert_{\mathcal{C}^{0} \rightarrow \mathcal{C}^{\gamma}} \Vert G(\bar{u}(s)) \Vert_{\mathcal{C}^{0}} \Vert D_\varepsilon \bar{v}(s) - D_\varepsilon v_{\varepsilon}(s) \Vert_{\mathcal{C}^{0}} \,ds \label{eq:Psi_3}\\
&\lesssim \Vert \bar{v} - v_{\varepsilon} \Vert_{\mathcal{C}^{1}_{(1-\alpha)/2, t_\varepsilon}} \int_{0}^{t_\varepsilon} (t_\varepsilon-s)^{-\frac{\gamma}{2}} s^{\frac{\alpha - 1}{2}} \,ds \lesssim {t_\varepsilon}^{\frac{1 + \alpha - \gamma}{2}} \Vert \bar{v} - v_{\varepsilon} \Vert_{\mathcal{C}^{1}_{(1-\alpha)/2, t_\varepsilon}}\;,
\end{equs}
where we have used boundedness of $\Vert \bar{u} \Vert_{\mathcal{C}^{0}_{t_\varepsilon}}$.

To bound $J_4$ we note that
\begin{equs}
\Vert U(s) - U_{\varepsilon}(s) \Vert_{\mathcal{C}^{1}} &\leq \Vert S_s \big( u^0 - u_\varepsilon^0 \big) \Vert_{\mathcal{C}^{1}} + \Vert S_s \left( X(0) - X_\varepsilon(0) \right) \Vert_{\mathcal{C}^{1}} \\
&\quad+ \Vert \big(S_s - S^{(\eps)}_s\big) \left( u_\varepsilon^0 - X_\varepsilon(0) \right) \Vert_{\mathcal{C}^{1}}\\
&\lesssim s^{\frac{\alpha - 1}{2}} \left(\Vert u^0 - u_\varepsilon^0 \Vert_{\mathcal{C}^{\alpha}} + \Vert X(0) - X_\varepsilon(0) \Vert_{\mathcal{C}^{\alpha}} \right) \\
&\quad+ s^{-\frac{1}{2}} \varepsilon^{\alpha_\star - \kappa} \left(\Vert u_\varepsilon^0 \Vert_{\mathcal{C}^{\alpha_\star}} + \Vert X_\varepsilon(0) \Vert_{\mathcal{C}^{\alpha_\star}} \right),\label{eq:initial_eps}
\end{equs}
for any $\kappa > 0$ sufficiently small. Here, in the last estimate we used Lemma \ref{lem:S_S_eps} with $\lambda = \alpha_\star - \kappa$. Using this estimate and \eqref{eq:D_eps_bound} we obtain
\begin{equs}
\Vert J_4 \Vert_{\mathcal{C}^{\gamma}} &\leq \int_{0}^{t_\varepsilon} \Vert S_{t_\varepsilon-s} \Vert_{\mathcal{C}^{0} \rightarrow \mathcal{C}^{\gamma}} \Vert G(\bar{u}(s)) \Vert_{\mathcal{C}^{0}} \Vert D_\varepsilon U(s) - D_\varepsilon U_{\varepsilon}(s) \Vert_{\mathcal{C}^{0}} \,ds \\
&\lesssim \int_{0}^{t_\varepsilon} (t_\varepsilon-s)^{-\frac{\gamma}{2}} \Vert U(s) - U_{\varepsilon}(s) \Vert_{\mathcal{C}^{1}} \,ds\\ &\lesssim t_\varepsilon^{\frac{1 + \alpha - \gamma}{2}} \left(\Vert u^0 - u_\varepsilon^0 \Vert_{\mathcal{C}^{\alpha}} + \Vert X(0) - X_\varepsilon(0) \Vert_{\mathcal{C}^{\alpha}} \right) + \varepsilon^{\alpha_\star - \kappa}\;.\label{eq:Psi_4}
\end{equs}

Exploiting continuous differentiability of the function $G$ we get
\begin{equs}
\Vert J_5 \Vert_{\mathcal{C}^{\gamma}} &\leq \int_{0}^{t_\varepsilon} \Vert S_{t_\varepsilon-s} \Vert_{\mathcal{C}^{0} \rightarrow \mathcal{C}^{\gamma}} \Vert G(\bar{u}(s)) - G(u_\varepsilon(s))\Vert_{\mathcal{C}^{0}} \Vert D_\varepsilon v_{\varepsilon}(s) + D_\varepsilon U_{\varepsilon}(s) \Vert_{\mathcal{C}^{0}} \,ds\\
&\lesssim \int_{0}^{t_\varepsilon} (t_\varepsilon-s)^{- \frac{\gamma}{2}} s^{\frac{\alpha_\star - 1 - \kappa}{2}} \Vert \bar{u}(s) - u_\varepsilon(s)\Vert_{\mathcal{C}^{0}} \,ds \label{eq:Psi_5}\\
&\lesssim t_\varepsilon^{\frac{1 + \alpha_\star - \gamma - \kappa}{2}} \Vert \bar{v} - v_{\varepsilon} \Vert_{\mathcal{C}^{\alpha}_{t_\varepsilon}} + \Vert X - X_{\varepsilon} \Vert_{\mathcal{C}^{\alpha}_{t_\varepsilon}} + \Vert u^0 - u_\varepsilon^0 \Vert_{\mathcal{C}^{\alpha}} + \varepsilon^{\alpha_\star - \kappa}\;,
\end{equs}
where in the second line we have used a bound, similar to \eqref{eq:initial_eps},
\begin{equ}[eq:D_eps_U_eps_bound]
\Vert D_\varepsilon U_{\varepsilon}(s)\Vert_{\mathcal{C}^{0}} \lesssim \Vert U_{\varepsilon}(s) \Vert_{\mathcal{C}^{1}} \lesssim s^{\frac{\alpha_\star - 1 - \kappa}{2}}\;.
\end{equ}
Moreover, in the estimate \eqref{eq:Psi_5} we have used the bound
\begin{equs}[eq:U_U_eps_0]
\Vert U(s) - U_{\varepsilon}(s) \Vert_{\mathcal{C}^{0}} \lesssim \Vert u^0 - u_\varepsilon^0 \Vert_{\mathcal{C}^{0}} + \Vert X(0) - X_\varepsilon(0) \Vert_{\mathcal{C}^{0}} + \varepsilon^{\alpha_\star - \kappa}\;,
\end{equs}
which is obtained in a way similar to \eqref{eq:initial_eps}.

Using Lemma \ref{lem:S_S_eps}, the integral $J_6$ can be bounded by
\begin{equs}
\Vert J_6 \Vert_{\mathcal{C}^{\gamma}} &\leq \int_{0}^{t_\varepsilon} \Vert S_{t_\varepsilon-s} - S^{(\eps)}_{t_\varepsilon-s} \Vert_{\mathcal{C}^{0} \rightarrow \mathcal{C}^{\gamma}} \Vert G(u_\varepsilon(s)) \Vert_{\mathcal{C}^{0}} \Vert D_\varepsilon v_{\varepsilon}(s) + D_\varepsilon U_{\varepsilon}(s) \Vert_{\mathcal{C}^{0}} \,ds\\
&\lesssim \varepsilon^{\alpha_\star - \kappa} \int_{0}^{t_\varepsilon} (t_\varepsilon-s)^{ - \frac{\alpha_\star + \gamma - \kappa/2}{2}} s^{\frac{\alpha_\star - 1 - \kappa/2}{2}} \,ds \lesssim t_\varepsilon^{\frac{1-\gamma}{2}} \varepsilon^{\alpha_\star - \kappa}\;,\label{eq:Psi_6}
\end{equs}
where we have used the bound \eqref{eq:D_eps_U_eps_bound}.

Combining the bounds \eqref{eq:Psi_1} -- \eqref{eq:Psi_6} we obtain the claimed estimate \eqref{eq:Psi_Psi_eps}.
\end{proof}

In the following proposition we provide a bound on the terms $\mathbf{F}^{\bar{v}}$ and $\mathbf{F}_{\varepsilon}^{v_{\varepsilon}}$ defined in \eqref{eq:u_mild_terms} and \eqref{eq:Phi_Psi_eps} respectively.

\begin{proposition}
\label{prop:Phi_Phi_eps}
For any $\gamma \in (0,1]$ and $\kappa > 0$ small enough the following bound holds
\begin{equ}
\Vert \mathbf{F}^{\bar{v}}(t_\varepsilon) - \mathbf{F}_{\varepsilon}^{v_{\varepsilon}}(t_\varepsilon) \Vert_{\mathcal{C}^{\gamma}} \lesssim t_\varepsilon^{1 - \frac{\gamma}{2}} \Vert \bar{v} - v_{\varepsilon} \Vert_{\mathcal{C}^{0}_{t_\varepsilon}} + \Vert X - X_{\varepsilon} \Vert_{\mathcal{C}^{0}_{t_\varepsilon}} + \Vert u^0 - u_{\varepsilon}^0 \Vert_{\mathcal{C}^{0}} + \varepsilon^{\alpha_\star - \kappa}\;.
\end{equ}
\end{proposition}
\begin{proof}
Using continuous differentiability of the function $F$, Lemma \ref{lem:S_S_eps} and recalling that $\bar{u} = \bar{v} + X + U$ we get
\begin{equs}
\Vert \mathbf{F}^{\bar{v}}(t_\varepsilon) &- \mathbf{F}^{v_{\varepsilon}}_{\varepsilon}(t_\varepsilon) \Vert_{\mathcal{C}^{\gamma}} \leq \int_{0}^{t_\varepsilon} \Vert S_{t_\varepsilon-s} \Vert_{\mathcal{C}^{0} \rightarrow \mathcal{C}^{\gamma}} \Vert F(\bar{u}(s)) - F(u_\varepsilon(s))\Vert_{\mathcal{C}^{0}} \,ds \\
&\quad+ \int_{0}^{t_\varepsilon} \Vert S_{t_\varepsilon-s} - S^{(\eps)}_{t_\varepsilon-s} \Vert_{\mathcal{C}^{0} \rightarrow \mathcal{C}^{\gamma}} \Vert F(u_\varepsilon(s)) \Vert_{\mathcal{C}^{0}} \,ds\\
&\lesssim \int_0^{t_\varepsilon} (t_\varepsilon-s)^{-\frac{\gamma}{2}} \Vert \bar{u}(s) - u_{\varepsilon}(s) \Vert_{\mathcal{C}^{0}} \,ds + \varepsilon^{\frac{1}{2} - \kappa} \int_0^{t_\varepsilon} (t_\varepsilon-s)^{-\frac{1}{4} - \frac{\gamma}{2}} \,ds \\
&\lesssim t_\varepsilon^{1 - \frac{\gamma}{2}} \Vert \bar{v} - v_{\varepsilon} \Vert_{\mathcal{C}^{0}_{t_\varepsilon}} + \Vert X - X_{\varepsilon} \Vert_{\mathcal{C}^{0}_{t_\varepsilon}} + \Vert u^0 - u_{\varepsilon}^0 \Vert_{\mathcal{C}^{0}} + \varepsilon^{\alpha_\star - \kappa}\;.
\end{equs}
Here, we have used boundedness of $\Vert u_\varepsilon \Vert_{\mathcal{C}^0_{t_\varepsilon}}$ and the estimate \eqref{eq:U_U_eps_0}.
\end{proof}

The following lemma shows how the processes \eqref{eq:D_eps_X_def} behave in the supremum norm. In particular,
it shows that they converge to $0$ as soon as $|w| > 2$.

\begin{lemma}
\label{lem:X_gamma_eps}
For any word $w \in \mathcal{A}_p$, the bound
\begin{equ}
\sup_{s \in [0,t_\varepsilon]} \Vert \angles{D_\varepsilon \mathbf{X}_{\varepsilon}(s;\cdot)}{e_w} \Vert_{\mathcal{C}^{0}} \lesssim \varepsilon^{|w|\alpha_{\star}-1}\;,
\end{equ}
holds uniformly in $\eps$ and $t$.
\end{lemma}
\begin{proof}
Since $\mathbf{X}_{\varepsilon}(s)$ is a rough path of regularity $\alpha_\star$, we can use the third property in Definition \ref{def:RP} to get
\begin{equs}
\left| \angles{ D_{\varepsilon} \mathbf{X}_{\varepsilon}(s;x)}{e_w} \right| &\leq \frac{1}{\varepsilon} \int_{\real} \left|  \angles{ \mathbf{X}_{\varepsilon}(s;x,x+\varepsilon z)}{e_w} \right| |\mu|(dz) \\
&\lesssim \varepsilon^{|w|\alpha_{\star}-1} \int_{\real} |z|^{|w|\alpha_{\star}} |\mu|(dz) \lesssim \varepsilon^{|w|\alpha_{\star}-1}\;.
\end{equs}
Here, we have used the assumption on the moments of $|\mu|$.
\end{proof}

In the following proposition we obtain a bound on the term $\mathbf{\bar H}_{\varepsilon}^{v_{\varepsilon}}$ defined in \eqref{eq:Upsilon_eps}.

\begin{proposition}
\label{prop:U_bar}
For any $\gamma \in (0,1]$ we have the estimate $\Vert \mathbf{\bar H}_{\varepsilon}^{v_{\varepsilon}} \Vert_{\mathcal{C}_{t_\varepsilon}^{\gamma}} \lesssim \varepsilon^{3\alpha_{\star}-1}$.
\end{proposition}
\begin{proof}
We use Lemma \ref{lem:S_eps} to estimate the approximate heat semigroup, and Lemma \ref{lem:X_gamma_eps}:
\begin{equs}
\Vert \mathbf{\bar H}_{\varepsilon}^{v_{\varepsilon}}&(t_\varepsilon) \Vert_{\mathcal{C}^{\gamma}} \\
&\lesssim \sum_{\substack{w \in \mathcal{A}_{p-1} \\ |w| \geq 2}} \int_{0}^{t_\varepsilon} (t_\varepsilon-s)^{-\frac{\gamma}{2} -\kappa} \Vert D^{w} G(u_{\varepsilon}(s)) \Vert_{\mathcal{C}^{0}}  \Vert \angles{D_\varepsilon \mathbf{X}_{\varepsilon}(s;\cdot)}{e_w \otimes e_1} \Vert_{\mathcal{C}^{0}} \,ds\\
&\lesssim \sum_{\substack{w \in \mathcal{A}_{p-1} \\ |w| \geq 2}} t_\varepsilon^{1-\frac{\gamma}{2} -\kappa} \varepsilon^{(|w|+1)\alpha_{\star}-1} \lesssim \varepsilon^{3\alpha_{\star}-1}\;,
\end{equs}
for $\kappa > 0$ small enough. This is the claimed bound.
\end{proof}


\section{Convergence of the correction term}
\label{sec:correction}

In this section we show that the term $\mathbf{H}_{\varepsilon}^{v_{\varepsilon}}$, defined in \eqref{eq:Upsilon_eps}, converges to the correction term $\mathbf{H}^{\bar{v}}$ from \eqref{eq:u_mild_terms}. In view of Remark~\ref{rem:stop_time}, we only consider time intervals up to the stopping time $\varrho_{K,\varepsilon}$, by using the notation \eqref{e:StoppedTime}.

To shorten the notation we define $\mathbb{X}_\varepsilon(t)$ to be the projection of the rough path $\mathbf{X}_\varepsilon(t)$ to the second level of the tensor algebra. The following lemma is similar to \cite[Proposition 4.1]{HMW14}, but the bound is in a H\"{o}lder norm rather than a Sobolev norm.
\begin{lemma}
\label{lem:DX_convergence}
For any $\gamma \in \bigl(0, \frac{1}{2}\bigr)$, any $t > 0$ and any $\kappa > 0$ small enough we have
\begin{equ}
\mathbb{E} \Big[ \sup_{s \in [0,t]} \left\Vert D_\varepsilon \mathbb{X}_{\varepsilon}(s, \cdot) - \Lambda \mathrm{Id} \right\Vert_{\mathcal{C}^{-\gamma}} \Big] \lesssim \varepsilon^{\gamma - \kappa}\;.
\end{equ}
\end{lemma}
\begin{proof}
The proof is almost identical to that of \cite[Proposition 4.1]{HMW14}, but we use Lemma~\ref{lem:Kolmogorov} to reduce oneself to moment bounds on the Paley-Littlewood blocks
of $D_\varepsilon \mathbb{X}_{\varepsilon}$, instead of using pointwise bounds.
\end{proof}

A bound on $\mathbf{H}^{\bar{v}}$ and $\mathbf{H}^{v_{\varepsilon}}_\varepsilon$, defined in \eqref{eq:u_mild_terms} and \eqref{eq:Upsilon_eps} respectively, is given in the next proposition. 
\begin{proposition}
\label{prop:U_U_eps}
For any $\gamma \in (0, 1]$ and any $\kappa > 0$ sufficiently small we have
\begin{equs}
\mathbb{E} \Vert \mathbf{H}^{\bar{v}} (t_\varepsilon)- \mathbf{H}^{v_{\varepsilon}}_\varepsilon(t_\varepsilon) \Vert_{\mathcal{C}^{\gamma}} \lesssim T^{1 - \frac{\gamma}{2}} \mathbb{E} \Vert \bar{v} - v_\varepsilon \Vert_{\mathcal{C}^{0}_{t_\varepsilon}} + \mathbb{E} \Vert X - X_\varepsilon \Vert_{\mathcal{C}^{0}_{t_\varepsilon}} \\
+ \mathbb{E} \Vert u^0 - u_{\varepsilon}^0 \Vert_{\mathcal{C}^{0}} + \varepsilon^{\alpha_\star - \kappa}\;,
\end{equs}
where $T > 0$ is as in the definition of the stopping times above \eqref{e:StoppedTime} and the constant $\alpha_\star$ is defined in the beginning of Section~\ref{ss:APriori}.
\end{proposition}
\begin{proof}
Let us define the functions $\mathcal{F}(u)_i = \Lambda ~\mathrm{div} G_i(u)$ and
\begin{equs}
\mathcal{F}_\varepsilon(u)_i(s,x) &= \sum_{w \in \mathcal{A}} D^{w} G_{ij}(u(s,x)) \angles{D_\varepsilon \mathbf{X}_{\varepsilon}(s,x)}{e_w \otimes e_j}\;,
\end{equs}
where as usual the sum over $j$ is omitted. Then we can write
\begin{equs}
\mathbf{H}^{\bar{v}}&(t_\varepsilon) - \mathbf{H}^{v_\varepsilon}_\varepsilon(t_\varepsilon) = \int_0^{t_\varepsilon} S_{t_\varepsilon-s} \Big(\mathcal{F}(u_\varepsilon) - \mathcal{F}_\varepsilon(u_\varepsilon)\Big)(s) \,ds\\
&\quad + \int_0^{t_\varepsilon} S_{t_\varepsilon-s} \Big(\mathcal{F}(\bar{u}) - \mathcal{F}(u_\varepsilon)\Big)(s) \,ds + \int_0^{t_\varepsilon} \Big(S_{t_\varepsilon-s} - S^{(\varepsilon)}_{t_\varepsilon-s}\Big) \mathcal{F}_\varepsilon(u_\varepsilon)(s) \,ds \\
&=: J_1 + J_2 + J_3\;.
\end{equs}
To bound $J_1$ we note that we can rewrite
\begin{equ}
\left(\mathcal{F}(u_\varepsilon) - \mathcal{F}_\varepsilon(u_\varepsilon)\right)_i(s,x) = \sum_{w \in \mathcal{A}} D^{w} G_{ij}(u_\varepsilon(s,x)) \left(\Lambda \delta_{w, j} - \angles{D_\varepsilon \mathbf{X}_{\varepsilon}(s,x)}{e_w \otimes e_j} \right).
\end{equ}
Therefore, applying Lemma \ref{lem:S_reg} with $\eta \in (0, \alpha_\star)$ and Lemma \ref{lem:multipl}, we obtain
\begin{equs}
\Vert J_1 \Vert_{\mathcal{C}^{\gamma}} &\lesssim \int_0^{t_\varepsilon} \Vert S_{t_\varepsilon-s} \Vert_{\mathcal{C}^{-\eta} \rightarrow \mathcal{C}^{\gamma}} \Vert \left(\mathcal{F}(u_\varepsilon) - \mathcal{F}_\varepsilon(u_\varepsilon)\right)(s) \Vert_{\mathcal{C}^{-\eta}} \,ds\\
&\lesssim \sup_{s \in [0,t_\varepsilon]} \left\Vert D_\varepsilon \mathbb{X}_{\varepsilon}(s, \cdot) - \Lambda \mathrm{Id} \right\Vert_{\mathcal{C}^{-\eta}} \Vert DG(u_\varepsilon) \Vert_{\mathcal{C}^{\alpha_\star}_{t_\varepsilon}} \int_0^{t_\varepsilon} (t_\varepsilon-s)^{-\frac{\eta + \gamma}{2}} \,ds\;.
\end{equs}
That gives us, using the boundedness of $\Vert u_\varepsilon \Vert_{\mathcal{C}_{t_\varepsilon}^{\alpha_\star}}$ and Lemma \ref{lem:DX_convergence},
\begin{equ}[eq:U_U_eps_1]
\mathbb{E}\Vert J_1 \Vert_{\mathcal{C}^{\gamma}_{t_\varepsilon}} \lesssim T^{1-\frac{\eta + \gamma}{2}} \varepsilon^{\eta - \kappa}\;.
\end{equ}
A bound on $J_2$ follows from Lemma \ref{lem:S_reg} and regularity of $G$,
\begin{equs}[eq:U_U_eps_2]
\Vert J_2 \Vert_{\mathcal{C}^{\gamma}} &\lesssim \int_0^{t_\varepsilon} (t_\varepsilon-s)^{- \frac{\gamma}{2}} \Vert \mathcal{F}(\bar{u}(s)) - \mathcal{F}(u_\varepsilon(s)) \Vert_{\mathcal{C}^{0}} \,ds \\
&\lesssim \int_0^{t_\varepsilon} (t_\varepsilon-s)^{- \frac{\gamma}{2}} \Vert \bar{u}(s) - u_\varepsilon(s) \Vert_{\mathcal{C}^{0}} \,ds \\
&\lesssim t_\varepsilon^{1 - \frac{\gamma}{2}} \Vert \bar{v} - v_\varepsilon \Vert_{\mathcal{C}^{0}_{t_\varepsilon}} + \Vert X - X_\varepsilon \Vert_{\mathcal{C}^{0}_{t_\varepsilon}} + \Vert u^0 - u_{\varepsilon}^0 \Vert_{\mathcal{C}^{0}} + \varepsilon^{\alpha_\star - \kappa}\;.
\end{equs}
Here, we have used the representation of $\bar{u}$ via $\bar{v}$ and the bound \eqref{eq:U_U_eps_0}.

For the third term we use Lemma \ref{lem:S_S_eps} with $\lambda = \frac{1}{2} - \kappa$,
\begin{equs}[eq:U_U_eps_3]
\Vert J_3 \Vert_{\mathcal{C}^{\gamma}_{t_\varepsilon}} \lesssim \varepsilon^{\frac{1}{2} - \kappa} t_\varepsilon^{1-\frac{1}{2}\left(\gamma +\frac{1}{2}\right)} \Vert \mathcal{F}_\varepsilon(u_\varepsilon) \Vert_{\mathcal{C}^{0}_{t_\varepsilon}} \lesssim \varepsilon^{\frac{1}{2} - \kappa} t_\varepsilon^{\frac{3}{4} - \frac{\gamma}{2}}\;,
\end{equs}
where we have used boundedness of the second-order iterated integral $\mathbb{X}_\varepsilon$ and $\Vert u_\varepsilon \Vert_{\mathcal{C}_{t_\varepsilon}^\alpha}$. Combining the estimates \eqref{eq:U_U_eps_1}, \eqref{eq:U_U_eps_2} and \eqref{eq:U_U_eps_3} we obtain the claimed bound.
\end{proof}


\section{Estimates on rough terms}
\label{sec:rough_int}

In this section we obtain bounds on the terms involving rough integrals. As usual, we will use the notation \eqref{e:StoppedTime}, which in view of Remark \ref{rem:stop_time} means that all the quantities involved in the definition of $\varrho_{K,\varepsilon}$ are bounded. Furthermore, let us define the quantity
\begin{equs}[eq:D_eps]
\mathcal{D}_\varepsilon(t_\varepsilon) &:= \Vert X- X_\varepsilon \Vert_{\mathcal{C}^{0}_{t_\varepsilon}} + \VERT \mathbf{X} - \mathbf{X}_\varepsilon \VERT_{\alpha, t_\varepsilon} + \Vert \bar{v}-v_\varepsilon \Vert_{\mathcal{C}^{\alpha}_{t_\varepsilon}} \\
&\quad + \Vert \bar{v}-v_\varepsilon \Vert_{\mathcal{C}^{1}_{(1-\alpha)/2, t_\varepsilon}} + \Vert u^0 - u^0_\varepsilon \Vert_{\mathcal{C}^{\alpha}}\;,
\end{equs}
where the norm $\VERT \cdot \VERT_{\alpha, t_\varepsilon}$ was introduced in \eqref{eq:RPNorm_time}.

The next lemma provides bounds on the rough integrals $Z$ and $Z_\varepsilon$ defined in \eqref{eq:Z_def} and \eqref{eq:Z_eps_def} respectively.

\begin{lemma}
For $t>0$ we have the following results
\begin{equs}
\label{eq:Z_0}
\Vert Z(t_\varepsilon)\Vert_{\mathcal{C}^{\alpha_{\star}}} &\lesssim t_\varepsilon^{-\frac{\alpha_{\star}}{2}}\;,\\
\label{eq:Z_0_Z_eps}
Z(t_\varepsilon) - Z_{\varepsilon}(t_\varepsilon) &= T_1(t_\varepsilon) + T_2(t_\varepsilon)\;,
\end{equs}
where, for $\kappa > 0$ small enough, the bounds
\begin{equs}
\Vert T_1(t_\varepsilon)\Vert_{\mathcal{C}^{\alpha}} \lesssim t_\varepsilon^{\frac{\alpha-1}{2}} \left(\mathcal{D}_{\varepsilon}(t_\varepsilon) + \varepsilon^{\alpha_\star - \alpha - \kappa}\right), \qquad \Vert T_2(t_\varepsilon)\Vert_{\mathcal{C}^{\alpha_\star}} \lesssim \varepsilon^{3\alpha_{\star}-1} t_\varepsilon^{-\frac{\alpha_\star}{2}}\;,
\end{equs}
hold with $\mathcal{D}_{\varepsilon}$ defined in \eqref{eq:D_eps}.
\end{lemma}
\begin{proof}
Since $\bar{u}(s) - X(s) \in \mathcal{C}^1$, for $s \leq t_\varepsilon$, the process $Y_{ij}(s) = G_{ij}(\bar{u}(s))$ is controlled by the $\alpha_\star$-regular rough path $\mathbf{X}(s)$ with the rough path derivative $Y'_{ij}(s) = DG_{ij}(\bar{u}(s))$ and the remainder
\begin{equs}
&R_{Y_{ij}}(s; x,y) = DG_{ij}(\bar{u}(s,x)) \left( \bar{v}(s; x, y) + U(s; x, y) \right) \\
 &+ \int_0^1 \left(DG_{ij}(\lambda \bar{u}(s,y) + (1-\lambda) \bar{u}(s,x)) - DG_{ij}(\bar{u}(s,x)) \right) \bar{u}(s; x, y)\, d\lambda\;,
\end{equs}
where we use the notation $\bar{v}(s; x, y) = \bar{v}(s, y) - \bar{v}(s, y)$ and respectively for $U$ and $\bar{u}$. Here, by the rough path derivative we mean the projection of the controlled rough path on $(\real^n)^*$ in Definition \ref{def:RP_control}, and the remainder is a collection of all the processes $R_Y^w$ from \eqref{eq:Y_expansion}.

From the regularity assumptions for the function $G$ and the processes $\bar{u}$ and $\bar{v}$, we obtain the bounds
\begin{equs}[eq:Y_bounds]
\Vert Y_{ij}(s) \Vert_{\mathcal{C}^{\alpha_\star}} \lesssim 1\;,\quad \Vert Y'_{ij}(s) \Vert_{\mathcal{C}^{\alpha_\star}} \lesssim 1\;,\quad \Vert R_{Y_{ij}}(s) \Vert_{\mathcal{B}^{2 \alpha_\star}} \lesssim s^{-\frac{\alpha_{\star}}{2}}\;.
\end{equs}
The power of $s$ in the last estimate comes from the bound $\Vert U(s) \Vert_{2\alpha_\star} \lesssim s^{-\frac{\alpha_{\star}}{2}}$, which is a consequence of Lemma \ref{lem:S_reg}. The estimate \eqref{eq:Z_0} follows from \eqref{eq:RI_bound} and \eqref{eq:Y_bounds}.

Similarly, for $s \leq t_\varepsilon$, the process $Y_{\varepsilon, ij}(s) = G_{ij}(u_\varepsilon(s))$ is controlled by the $\alpha_\star$-regular rough path $\mathbf{X}_\varepsilon(s)$ with the rough path derivative $Y'_{\varepsilon, ij}(s) = DG_{ij}(u_\varepsilon(s))$ and the remainder $R_{Y_{\varepsilon,ij}}(s)$, such that the following bounds hold
\begin{equs}[eq:Y_eps_bounds]
\Vert Y_{\varepsilon, ij}(s) \Vert_{\mathcal{C}^{\alpha_\star}} \lesssim 1\;,\quad \Vert Y'_{\varepsilon, ij}(s) \Vert_{\mathcal{C}^{\alpha_\star}} \lesssim 1\;,\quad \Vert R_{Y_{\varepsilon, ij}}(s) \Vert_{\mathcal{B}^{2 \alpha_\star}} \lesssim s^{-\frac{\alpha_{\star}}{2}}\;.
\end{equs}

To prove the bound \eqref{eq:Z_0_Z_eps}, we consider the processes $\bar{u}(s)$ and $u_\varepsilon(s)$ to be of H\"{o}lder regularity $\alpha$. Then they are controlled by the $\alpha$-regular rough paths $\mathbf{X}(s)$ and $\mathbf{X}_\varepsilon(s)$ respectively. Hence, we can extend $G_{ij}(\bar{u}(s))$ to the process $\mathcal{G}_{ij}(s): \mathbb{T} \rightarrow \big(T^{(p-1)}\big(\real^n\big)\big)^*$ which is controlled by $\mathbf{X}(s)$ as well and such that
\[
\angles{\mathcal{G}_{ij}(s,x)}{e_w} = D^{w} G_{ij}(\bar{u}(s,x))\;,
\]
for $w \in \mathcal{A}_{p-1}$. Then, as it was noticed in Subsection \ref{ssec:rough_int}, for every $w \in \mathcal{A}_{p-1}$ the following expansion holds
\begin{equs}
&\angles{\mathcal{G}_{ij}(s,y)}{e_w} - \angles{\mathcal{G}_{ij}(s,x)}{e_w} \\
=& \sum_{\bar{w} \in \mathcal{A}_{p-|w|-1} \setminus \emptyset} C_{\bar{w}} \angles{\mathcal{G}_{ij}(s,x)}{e_{\bar{w}} \otimes e_w} \angles{\mathbf{X}(s;x,y)}{e_{\bar{w}}} + R_{\mathcal{G}_{ij}}^w(s; x,y)\;.
\end{equs}
For any word $w \in \mathcal{A}_{p-1}$, the assumptions on $G$ and $\bar{u}$ imply $\Vert \angles{\mathcal{G}_{ij}(s)}{e_w} \Vert_{\mathcal{C}^{\alpha}} \lesssim 1$. Furthermore, from the argument of Subsection \ref{ssec:rough_int}, it is not difficult to obtain the estimate on the remainder: $\Vert R^w_{\mathcal{G}_{ij}}(s) \Vert_{\mathcal{B}^{(p - |w|) \alpha}} \lesssim s^{\frac{\alpha_{\star} - 1}{2}}$. The latter bound follows from $|\bar{u}(s; x,y)_{\bar{w}}| \lesssim |y-x|^{(p-|w|) \alpha}$, for any word $\bar{w}$ such that $|\bar{w}| = p-|w|$, and
\begin{equs}
| \bar{u}(s;x,y)_{\bar{w}} &- X(s;x,y)_{\bar{w}} | \lesssim \left| \bar{u}(s;x,y) - X(s;x,y) \right| \\
&\lesssim |y-x| \left( \Vert \bar{v}(s)\Vert_{\mathcal{C}^{1}} + \Vert U(s)\Vert_{\mathcal{C}^{1}}\right) \lesssim |y-x| \left( 1 + s^{\frac{\alpha_{\star} - 1}{2}}\right),
\end{equs}
for any word $\bar{w} \in \mathcal{A}_{p-|w|-1} \setminus \{\emptyset\}$. Here, in the last line we have used the bound
\[
\Vert U(s) \Vert_{\mathcal{C}^{1}} \lesssim s^{\frac{\alpha_\star - 1}{2}} \left(\Vert u^0 \Vert_{\mathcal{C}^{\alpha_\star}} + \Vert X(0) \Vert_{\mathcal{C}^{\alpha_\star}} \right),
\]
which follows from Lemma \ref{lem:S_reg}.

In the same way the process $G_{ij}(u_\varepsilon(s))$ can be extended to $\mathcal{G}^\varepsilon_{ij}(s): \mathbb{T} \rightarrow \big(T^{(p-1)}\big(\real^n\big)\big)^*$ which is controlled by $\mathbf{X}_\varepsilon(s)$. We denote the remainders by $R_{\mathcal{G}^\varepsilon_{ij}}^w$. Furthermore, the corresponding bounds hold
\[
\Vert \angles{\mathcal{G}^\varepsilon_{ij}(s)}{e_w} \Vert_{\mathcal{C}^{\alpha}} \lesssim 1\;,\qquad \Vert R^w_{\mathcal{G}_{ij}^\varepsilon}(s) \Vert_{\mathcal{B}^{(p - |w|) \alpha}} \lesssim s^{\frac{\alpha_{\star} - 1}{2}}\;,
\]
for any word $w \in \mathcal{A}_{p-1}$.

The following estimate follows from the regularity of the function $G$,
\begin{equs}
\Vert \angles{\mathcal{G}_{ij}(s) - \mathcal{G}^\varepsilon_{ij}(s)}{e_w} &\Vert_{\mathcal{C}^{\alpha}} \lesssim \Vert \bar{u}(s) - u_\varepsilon(s) \Vert_{\mathcal{C}^{\alpha}} \label{eq:G_G_eps}\\
\lesssim \Vert &X(s) - X_\varepsilon(s) \Vert_{\mathcal{C}^{\alpha}} + \Vert \bar{v}(s) - v_\varepsilon(s) \Vert_{\mathcal{C}^{\alpha}} + \Vert u^0 - u^0_\varepsilon \Vert_{\mathcal{C}^{\alpha}},
\end{equs}
where $w \in \mathcal{A}_{p-1}$. Furthermore, the following bound holds
\begin{equs}
|\bar{u}(s; x,y)_{\bar{w}} - u_\varepsilon(s; x,y)_{\bar{w}}| \lesssim |y-x|^{(p-|w|) \alpha} \Vert \bar{u}(s) - u_\varepsilon(s) \Vert_{\mathcal{C}^{\alpha}}\;,
\end{equs}
for a word $\bar{w}$ such that $|\bar{w}| = p-|w|$, and for any word $\bar{w} \in \mathcal{A}_{p-|w|-1} \setminus \{\emptyset\}$ one has
\begin{equs}
| \bar{u}(s;x,y)_{\bar{w}} &- X(s;x,y)_{\bar{w}} - u_\varepsilon(s;x,y)_{\bar{w}} + X_\varepsilon(s;x,y)_{\bar{w}} | \\
&\quad\lesssim \left| \bar{u}(s;x,y) - X(s;x,y) - u_\varepsilon(s;x,y) + X_\varepsilon(s;x,y) \right| \\
&\quad\lesssim |y-x| \left( \Vert \bar{v}(s) - v_\varepsilon(s)\Vert_{\mathcal{C}^{1}} + \Vert U(s) - U_\varepsilon(s)\Vert_{\mathcal{C}^{1}}\right)\\
&\quad\lesssim |y-x| s^{\frac{\alpha - 1}{2}} \left( \Vert \bar{v} - v_\varepsilon \Vert_{\mathcal{C}^{1}_{(1-\alpha)/2, s}} + \Vert X(s) - X_\varepsilon(s) \Vert_{\mathcal{C}^{\alpha}} \right.\\
&\qquad+\left. \Vert u^0 - u^0_\varepsilon \Vert_{\mathcal{C}^{\alpha}} + \varepsilon^{\alpha_\star - \alpha - \kappa}\right).
\end{equs}
Here, in the last line we have used the bound
\begin{equs}
\Vert U(s) - U_\varepsilon(s)\Vert_{\mathcal{C}^{1}} &\lesssim \Vert S_s (X(0) - X_\varepsilon(0) - u^0 + u^0_\varepsilon) \Vert_{\mathcal{C}^{1}} \\
&\qquad+ \Vert (S_s - S^{(\eps)}_s) (X_\varepsilon(0) - u^0_\varepsilon)\Vert_{\mathcal{C}^{1}}\\
&\lesssim s^{\frac{\alpha-1}{2}} \left( \Vert X(0) - X_\varepsilon(0) \Vert_{\mathcal{C}^{\alpha}} + \Vert u^0 - u^0_\varepsilon \Vert_{\mathcal{C}^{\alpha}} + \varepsilon^{\alpha_\star - \alpha - \kappa} \right),
\end{equs}
for any $\kappa > 0$ sufficiently small, which follows from Lemmas \ref{lem:S_reg} and \ref{lem:S_S_eps}. From these bounds and Subsection \ref{ssec:rough_int} we obtain
\begin{equs}[eq:R_R_eps]
\Vert R^w_{\mathcal{G}_{ij}}(s) - R^w_{\mathcal{G}^\varepsilon_{ij}}(s) \Vert_{\mathcal{B}^{(p - |w|) \alpha}} &\lesssim s^{\frac{\alpha-1}{2}} \left( \Vert \bar{v} - v_\varepsilon \Vert_{\mathcal{C}^{1}_{(1-\alpha)/2, s}} + \Vert X - X_\varepsilon \Vert_{\mathcal{C}^{\alpha}_s} \right.\\
&\quad +\left. \Vert u^0 - u^0_\varepsilon \Vert_{\mathcal{C}^{\alpha}} + \varepsilon^{\alpha_\star - \alpha - \kappa}\right).
\end{equs}

In order to prove \eqref{eq:Z_0_Z_eps}, we define
\begin{equs}
Q_i^{\varepsilon}(t_\varepsilon;x,y) &:= \dashint_{x}^{y} G_{ij}(u_{\varepsilon}(t_\varepsilon,z))\, d_z X^j_{\varepsilon}(t_\varepsilon,z) - G_{ij}(u_{\varepsilon}(t_\varepsilon,x)) X^j_{\varepsilon}(t_\varepsilon; x,y) \\
&\quad - \sum_{w \in \mathcal{A}} D^w G_{ij}(u_{\varepsilon}(t_\varepsilon,x)) \angles{\mathbf{X}_\varepsilon(t_\varepsilon;x,y)}{e_{w} \otimes e_j}\;,\\
T_i^{\varepsilon}(t_\varepsilon;x,y) &:= \sum_{\substack{w \in \mathcal{A}_{p-1} \\ |w| \geq 2}} C_w \angles{\mathcal{G}^\varepsilon_{ij}(t_\varepsilon,x)}{e_w} \angles{\mathbf{X}_\varepsilon(t_\varepsilon;x,y)}{e_w \otimes e_j}\;,
\end{equs}
where we have omitted as usual the sum over $j$. From \eqref{eq:RI_approx_bound}, \eqref{eq:Y_eps_bounds} and Definition \ref{def:RP} we obtain
\begin{equs}[eq:Q_eps_bound]
\Vert Q_i^{\varepsilon}(t_\varepsilon) \Vert_{\mathcal{B}^{3\alpha_\star}} \lesssim t_\varepsilon^{-\frac{\alpha_\star}{2}}\;, \qquad \Vert T_i^{\varepsilon}(t_\varepsilon) \Vert_{\mathcal{B}^{3\alpha_\star}} \lesssim 1\;.
\end{equs}
Next, we can rewrite $Z^i - Z^i_{\varepsilon}$ in the following way
\begin{equs}
Z^i(t_\varepsilon,x) &- Z^i_{\varepsilon}(t_\varepsilon,x) \\
&= \left( \dashint_{-\pi}^{x} G_{i}(u(t_\varepsilon,y)) \,d_{y} X(t_\varepsilon,y) - \dashint_{-\pi}^{x} G_{i}(u_{\varepsilon}(t_\varepsilon,y)) \,d_{y} X_{\varepsilon}(t_\varepsilon,y) \right)\\
&\quad+ \int_{\real} \dashint_{-\pi}^{-\pi + \varepsilon z} \frac{\varepsilon z - \pi - y}{\varepsilon} G_{i}(u_{\varepsilon}(t_\varepsilon,y))\, d_{y} X_{\varepsilon}(t_\varepsilon,y)\, \mu(dz) \\
&\quad+ \int_{\real} \dashint_{x}^{x + \varepsilon z} \frac{y - \varepsilon z - x}{\varepsilon} G_{i}(u_{\varepsilon}(t_\varepsilon,y)) \,d_{y} X_{\varepsilon}(t_\varepsilon,y) \,\mu(dz)\\
&\quad- \int_{\real} \int_{-\pi}^{x} \frac{Q_i^{\varepsilon}(t_\varepsilon;y,y+\varepsilon z)}{\varepsilon} \,dy \,\mu(dz)\\
&\quad+ \int_{\real} \int_{-\pi}^{x} \frac{T_i^{\varepsilon}(t_\varepsilon;y,y+\varepsilon z)}{\varepsilon} \,dy \,\mu(dz) =: \sum_{1 \leq j \leq 5} I_{j}(t_\varepsilon,x)\;.
\end{equs}
Here, we have used the Fubini-type result proved in \cite[Lemma 2.10]{HW10}.

To bound $I_1$ we apply \eqref{eq:RI_two_bound} and use the bounds \eqref{eq:G_G_eps}, \eqref{eq:R_R_eps},
\begin{equs}
\Vert I_{1}(t_\varepsilon) \Vert_{\mathcal{C}^\alpha} \lesssim t_\varepsilon^{\frac{\alpha-1}{2}}  \left( \mathcal{D}_{\varepsilon}(t_\varepsilon) + \varepsilon^{\alpha_\star - \alpha - \kappa} \right),
\end{equs}
where $\mathcal{D}_{\varepsilon}$ is defined in \eqref{eq:D_eps}. It follows from \eqref{eq:Q_eps_bound} that
\begin{equs}
\Vert I_{4}(t_\varepsilon) \Vert_{\mathcal{C}^1} \lesssim \int_{\real} |z|^{3\alpha_\star} \mu(dz) \,\varepsilon^{3\alpha_\star - 1} \Vert Q_i^{\varepsilon}(t_\varepsilon) \Vert_{\mathcal{B}^{3\alpha_\star}} \lesssim \varepsilon^{3\alpha_\star - 1} t_\varepsilon^{-\frac{\alpha_\star}{2}}\;.
\end{equs}
In the same way from the second bound in \eqref{eq:Q_eps_bound} we derive
\begin{equs}
\Vert I_{5}(t_\varepsilon) \Vert_{\mathcal{C}^1} \lesssim \int_{\real} |z|^{3\alpha_\star} \mu(dz)\, \varepsilon^{3\alpha_\star - 1} \Vert T_i^{\varepsilon}(t_\varepsilon) \Vert_{\mathcal{B}^{3\alpha_\star}} \lesssim \varepsilon^{3\alpha_\star - 1}\;.
\end{equs}

To bound the integral $I_3$ let us define $u_{x,z,\varepsilon}(t_\varepsilon,y) := u_\varepsilon(t_\varepsilon, \varepsilon y - \varepsilon z - x)$ and the rough path $\mathbf{X}_{x,z,\varepsilon}(t_\varepsilon;y, \bar{y}) := \mathbf{X}_{\varepsilon}(t_\varepsilon; \varepsilon y - \varepsilon z - x, \varepsilon \bar{y} - \varepsilon z - x)$. Then we can perform the change of variables $\bar{y} = (y - \varepsilon z - x)/\varepsilon$ in the integral $I_3$ and obtain
\[
I_3 = \int_{\real} \dashint_{-z}^{0} Y_{x,z,\varepsilon}(t_\varepsilon,\bar{y}) \,d_{\bar{y}} X_{x,z,\varepsilon}(t_\varepsilon, \bar{y})\, \mu(dz)\;,
\]
where $X_{x,z,\varepsilon}(t_\varepsilon, \bar{y}) - X_{x,z,\varepsilon}(t_\varepsilon, y)$ is the projection of $\mathbf{X}_{x,z,\varepsilon}(t_\varepsilon;y, \bar{y})$ onto $\real^n$ and
\[
Y_{x,z,\varepsilon}(t_\varepsilon,\bar{y}) := \bar{y} G_{i}(u_{x,z,\varepsilon}(t_\varepsilon,\bar{y}))\;.
\]
Taking into account the a priori bounds on $u_\varepsilon$, we obtain from \cite[Lemma 2.2]{Hai11a} that $Y_{x,z,\varepsilon}(t_\varepsilon)$ is controlled by $\mathbf{X}_{x,z,\varepsilon}(t_\varepsilon)$ with the rough path derivative 
\[
Y'_{x,z,\varepsilon}(t_\varepsilon,\bar{y}) := \bar{y} D G_{i}(u_{x,z,\varepsilon}(t_\varepsilon,\bar{y}))\,
\]
and the remainder $R_{Y_{x,z,\varepsilon}}(t_\varepsilon)$ such that
\[
\Vert Y_{x,z,\varepsilon}(t_\varepsilon) \Vert_{\mathcal{C}^{\alpha_\star}} \lesssim 1\;, \quad \Vert Y'_{x,z,\varepsilon}(t_\varepsilon) \Vert_{\mathcal{C}^{\alpha_\star}} \lesssim 1\;, \quad \Vert R_{Y_{x,z,\varepsilon}}(t_\varepsilon) \Vert_{\mathcal{B}^{2\alpha_\star}} \lesssim t_\varepsilon^{-\frac{\alpha_\star}{2}}\;.
\]
Hence, the following bound follows from Proposition \ref{prop:RP_integrals} and the simple estimate $\VERT \mathbf{X}_{x,z,\varepsilon}(t_\varepsilon) \VERT_{\alpha_\star} \leq \varepsilon^{\alpha_\star} \VERT \mathbf{X}_{\varepsilon}(t_\varepsilon) \VERT_{\alpha_\star}$:
\begin{equs}
\Vert I_3(t_\varepsilon) \Vert_{\mathcal{C}^{\alpha_\star}} \leq \int_{\real} &\left\Vert \dashint_{\cdot}^{0} Y_{x,z,\varepsilon}(t_\varepsilon,\bar{y}) \,d_{\bar{y}} X_{x,z,\varepsilon}(t_\varepsilon, \bar{y}) \right\Vert_{\mathcal{C}^{\alpha_\star}} \, |z|^{\alpha_\star}\, \mu(dz)\\
\lesssim \int_{\real} &\VERT \mathbf{X}_{x,z,\varepsilon}(t_\varepsilon) \VERT_{\alpha_\star} \left( \Vert Y_{x,z,\varepsilon}(t_\varepsilon) \Vert_{\mathcal{C}^{\alpha_\star}} + \Vert Y'_{x,z,\varepsilon}(t_\varepsilon) \Vert_{\mathcal{C}^{\alpha_\star}} \right.\\
&+ \left.\Vert R_{Y_{x,z,\varepsilon}}(t_\varepsilon) \Vert_{\mathcal{B}^{2\alpha_\star}} \right) \, |z|^{\alpha_\star}\, \mu(dz) \lesssim \varepsilon^{\alpha_\star} t_\varepsilon^{-\frac{\alpha_\star}{2}}\;.
\end{equs}
Here we have also used the bound on the $\alpha_\star$th moment of the measure $\mu$. Similarly, we can obtain the bound $\Vert I_2(t_\varepsilon) \Vert_{\mathcal{C}^{\alpha_\star}} \lesssim \varepsilon^{\alpha_\star} t_\varepsilon^{-\frac{\alpha_\star}{2}}$. 

Now we set $T_1=I_1$ and $T_2=I_2 + I_3 + I_4 + I_5$ and obtain the claim.
\end{proof}

In the following proposition we prove a bound on $\mathbf{Z}^{\bar{v}}$ and $\mathbf{Z}^{v_{\varepsilon}}_{\varepsilon}$ defined in \eqref{eq:u_mild_terms} and \eqref{eq:Xi_eps_def} respectively.

\begin{proposition}
\label{prop:Xi1}
For $\gamma \in (0,1]$ and $\kappa > 0$ small enough we have the estimate
\begin{equs}
\Vert \mathbf{Z}^{\bar{v}}(t_\varepsilon) - \mathbf{Z}^{v_{\varepsilon}}_{\varepsilon}(t_\varepsilon)\Vert_{\mathcal{C}^{\gamma}} \lesssim t_\varepsilon^{\alpha - \frac{1}{2}(\gamma + \kappa)} \left(\mathcal{D}_{\varepsilon}(t_\varepsilon) + \varepsilon^{\alpha_\star - \alpha - \kappa}\right),
\end{equs}
where $\mathcal{D}_{\varepsilon}$ is defined in \eqref{eq:D_eps}.
\end{proposition}
\begin{proof}
We can rewrite $\mathbf{Z}^{\bar{v}} - \mathbf{Z}^{v_{\varepsilon}}_{\varepsilon}$ in the following way
\begin{equs}
\mathbf{Z}^{\bar{v}}(t_\varepsilon) - \mathbf{Z}^{v_{\varepsilon}}_{\varepsilon}(t_\varepsilon) &= \int_{0}^{t_\varepsilon} \partial_{x}( S_{t_\varepsilon-s} - S^{(\eps)}_{t_\varepsilon-s}) Z(s) \,ds + \int_{0}^{t_\varepsilon} \partial_{x} S^{(\eps)}_{t_\varepsilon-s} ( Z(s) - Z_{\varepsilon}(s)) \,ds \\
&=: J_{1} + J_{2}\;.
\end{equs}
By (\ref{eq:Z_0}) and Lemma \ref{lem:S_S_eps} with $\lambda = \alpha_\star - \alpha - \kappa$ we obtain for any $\kappa > 0$ small enough
\begin{equs}
\Vert &J_{1}\Vert_{\mathcal{C}^{\gamma}} \lesssim \int_{0}^{t_\varepsilon} \Vert S_{t_\varepsilon-s} - S^{(\eps)}_{t_\varepsilon-s}\Vert_{\mathcal{C}^{\alpha_{\star}} \rightarrow \mathcal{C}^{1 + \gamma}} \Vert Z(s)\Vert_{\mathcal{C}^{\alpha_{\star}}} \,ds\\
&\lesssim \varepsilon^{\alpha_\star - \alpha - \kappa} \int_{0}^{t_\varepsilon} (t_\varepsilon-s)^{-\frac{1}{2}(1+\gamma - \alpha)} s^{-\frac{\alpha_{\star}}{2}} \,ds \lesssim t_\varepsilon^{\frac{1}{2}(1 - \gamma + \alpha - \alpha_\star)} \varepsilon^{\alpha_\star - \alpha - \kappa}\;.\label{eq:Xi_Xi_eps_1}
\end{equs}
The second term can be estimated using Lemma \ref{lem:S_eps} and \eqref{eq:Z_0_Z_eps} by
\begin{equs}
\Vert J_{2}\Vert_{\mathcal{C}^{\gamma}} &\lesssim \int_{0}^{t_\varepsilon} \Vert S^{(\eps)}_{t_\varepsilon-s}\Vert_{\mathcal{C}^{\alpha} \rightarrow \mathcal{C}^{1 + \gamma}}\Vert T_1(s) \Vert_{\mathcal{C}^{\alpha}} \,ds + \int_{0}^{t_\varepsilon} \Vert S^{(\eps)}_{t_\varepsilon-s}\Vert_{\mathcal{C}^{\alpha_\star} \rightarrow \mathcal{C}^{1 + \gamma}} \Vert T_2 (s) \Vert_{\mathcal{C}^{\alpha_\star}} \,ds\\
&\lesssim \int_{0}^{t_\varepsilon} (t_\varepsilon-s)^{-\frac{1}{2}(1+\gamma - \alpha + \kappa)} s^{\frac{\alpha-1}{2}} \left(\mathcal{D}_{\varepsilon}(s) + \varepsilon^{\alpha_\star - \alpha - \kappa}\right) ds\\
&\quad+ \varepsilon^{3\alpha_{\star}-1} \int_{0}^{t_\varepsilon} (t_\varepsilon-s)^{-\frac{1}{2}(1+\gamma - \alpha_\star + \kappa)} s^{-\frac{\alpha_\star}{2}} \,ds \\
&\lesssim t_\varepsilon^{\alpha - \frac{1}{2}(\gamma + \kappa)} \left(\mathcal{D}_{\varepsilon}(t_\varepsilon) + \varepsilon^{\alpha_\star - \alpha - \kappa}\right) + \varepsilon^{3\alpha_{\star}-1} t_\varepsilon^{\frac{1}{2}(1 - \gamma - \kappa)}\;.
\label{eq:Xi_Xi_eps_2}
\end{equs}
Combining \eqref{eq:Xi_Xi_eps_1} and \eqref{eq:Xi_Xi_eps_2} we obtain the claimed bound.
\end{proof}


\section{Convergence of the solutions of the approximate equations}
\label{sec:convergence}

With the results from the previous sections at hand, we can prove Theorem~\ref{thm:first}.

\begin{proof}[Proof of Theorem \ref{thm:first}]
For $\alpha > 0$ as in the beginning of this section we define $p = \lfloor 1/\alpha \rfloor$. From the derivation of the bounds below we will see how small the value of $\alpha$ must be. To make the notation shorter, we introduce the following norm
\[
\Vert \cdot \Vert_{\alpha, t} := \Vert \cdot \Vert_{\mathcal{C}^{\alpha}_{t}} + \Vert \cdot \Vert_{\mathcal{C}^{1}_{(1-\alpha)/2, t}}\;.
\]

Then, using the notation \eqref{e:StoppedTime}, we obtain from \eqref{eq:u_mild} and \eqref{eq:u_eps_mild_new} the bound
\begin{equs}[eq:v_v_eps]
\Vert \bar{v} - v_{\varepsilon} \Vert_{\alpha, t_\varepsilon} &\leq \Vert \mathbf{G}^{\bar{v}} - \mathbf{G}_{\varepsilon}^{v_{\varepsilon}} \Vert_{\alpha, t_\varepsilon} + \Vert \mathbf{F}^{\bar{v}} - \mathbf{F}^{v_{\varepsilon}}_{\varepsilon} \Vert_{\alpha, t_\varepsilon} + \Vert \mathbf{H}^{\bar{v}} - \mathbf{H}^{v_{\varepsilon}}_{\varepsilon} \Vert_{\alpha, t_\varepsilon} \\
&\quad+ \Vert \mathbf{\bar H}^{v_{\varepsilon}}_{\varepsilon} \Vert_{\alpha, t_\varepsilon} + \Vert \mathbf{Z}^{\bar{v}} - \mathbf{Z}^{v_\varepsilon}_{\varepsilon} \Vert_{\alpha, t_\varepsilon}\;.
\end{equs}
We consider only time periods $t < 1$, for larger times the claim can easily be obtained 
by iteration.
To find a bound on the first term in \eqref{eq:v_v_eps} we use the results of Section \ref{sec:reaction}. Applying Proposition \ref{prop:Psi_Psi_eps} with a small constant $\kappa = \alpha$ we get
\begin{equs}
\Vert \mathbf{G}^{\bar{v}} - \mathbf{G}_{\varepsilon}^{v_{\varepsilon}} \Vert_{\alpha, t_\varepsilon} &\lesssim t_\varepsilon^{\frac{1}{2}} \Vert \bar{v} - v_{\varepsilon} \Vert_{\alpha, t_\varepsilon} \label{eq:I_1} + \Vert X - X_{\varepsilon} \Vert_{\mathcal{C}^{\alpha}_{t_\varepsilon}} \\
+& \Vert u^0 - u^0_{\varepsilon} \Vert_{\mathcal{C}^{\alpha}} + \varepsilon^{\alpha_\star - \alpha}\;.
\end{equs}

In order to bound the second term in \eqref{eq:v_v_eps}, we use Proposition \ref{prop:Phi_Phi_eps} with $\kappa = \alpha$,
\begin{equs}
\Vert \mathbf{F}^{\bar{v}} - \mathbf{F}^{v_{\varepsilon}}_{\varepsilon} \Vert_{\alpha, t_\varepsilon} &\lesssim t_\varepsilon^{\frac{1 - \alpha}{2}} \Vert \bar{v} - v_{\varepsilon} \Vert_{\mathcal{C}^{0}_{t}} + \Vert X - X_{\varepsilon} \Vert_{\mathcal{C}^{0}_{t_\varepsilon}} \\
&\quad+ \Vert u^0 - u_{\varepsilon}^0 \Vert_{\mathcal{C}^{0}} + \varepsilon^{\alpha_\star - \alpha}\;.\label{eq:I_2}
\end{equs}

Applying Proposition \ref{prop:U_U_eps} with the parameter $\kappa = \alpha$, we bound the expectation of the third term in \eqref{eq:v_v_eps} by
\begin{equs}[eq:I_3]
\mathbb{E} \Vert \mathbf{H}^{\bar{v}} - \mathbf{H}^{v_{\varepsilon}}_{\varepsilon} \Vert_{\alpha, t_\varepsilon} \lesssim T^{\frac{1 - \alpha}{2}} \mathbb{E} \Vert \bar{v} - v_\varepsilon \Vert_{\mathcal{C}^{0}_{t_\varepsilon}} + \mathbb{E} \Vert X - X_\varepsilon \Vert_{\mathcal{C}^{0}_{t_\varepsilon}} \\
+ \mathbb{E} \Vert u^0 - u_{\varepsilon}^0 \Vert_{\mathcal{C}^{0}} + \varepsilon^{ \alpha_\star - \alpha}\;,
\end{equs}
where $T > 0$ is as above \eqref{e:StoppedTime}.

A bound on the fourth term in \eqref{eq:v_v_eps} is a straightforward application of Proposition \ref{prop:U_bar},
\begin{equ}[eq:I_4]
\Vert \mathbf{\bar H}^{v_{\varepsilon}}_{\varepsilon} \Vert_{\mathcal{C}^{\alpha}_{t_\varepsilon}} + \Vert \mathbf{\bar H}^{v_{\varepsilon}}_{\varepsilon} \Vert_{\mathcal{C}^{1}_{t_\varepsilon}} \lesssim \varepsilon^{3\alpha_{\star}-1}\;.
\end{equ}
Using Proposition \ref{prop:Xi1} with the small parameter $\kappa = \alpha / 2$ we can bound the last term in \eqref{eq:v_v_eps} by
\begin{equs}[eq:I_5]
\Vert \mathbf{Z}^{\bar{v}} - \mathbf{Z}^{v_\varepsilon}_{\varepsilon} \Vert_{\alpha, t_\varepsilon} \lesssim t_\varepsilon^{\frac{\alpha}{4}} \mathcal{D}_{\varepsilon}(t_\varepsilon) + \varepsilon^{\alpha_\star - 3 \alpha /2}\;,
\end{equs}
where $\mathcal{D}_\varepsilon$ is defined in \eqref{eq:D_eps}.

Combining the bounds \eqref{eq:v_v_eps}--\eqref{eq:I_5} together we obtain
\begin{equs}[eq:E_first]
\mathbb{E} \Vert \bar{v} - v_{\varepsilon} \Vert_{\alpha, t_\varepsilon} \lesssim T^{\frac{\alpha}{4}} \mathbb{E} \Vert \bar{v} - v_{\varepsilon} \Vert_{\alpha, t_\varepsilon} + \mathbb{E} \Vert u^{0} - u_{\varepsilon}^{0} \Vert_{\mathcal{C}^{\alpha}} + \mathbb{E} \Vert X - X_{\varepsilon} \Vert_{\mathcal{C}^{\alpha}_{t_\varepsilon}} \\
+ \mathbb{E} \VERT \mathbf{X} - \mathbf{X}_\varepsilon \VERT_{\alpha, t_\varepsilon} + \varepsilon^{\frac{1}{2} - 3\alpha}\;,
\end{equs}
where we have used $\alpha_\star = \frac{1}{2} - \alpha$. By Lemma \ref{lem:X_X_eps} we can bound the norms of the controlling processes,
\begin{equs}
\mathbb{E} \Vert X - X_{\varepsilon} \Vert_{\mathcal{C}^{\alpha}_{t_\varepsilon}} + \mathbb{E} \VERT \mathbf{X} - \mathbf{X}_\varepsilon \VERT_{\alpha, t_\varepsilon} \lesssim \varepsilon^{\frac{1}{2} - 2\alpha}\;.
\end{equs}
Furthermore, by choosing $T$ in \eqref{e:StoppedTime} small enough we can absorb the first term on the right-hand side of \eqref{eq:E_first} into the left-hand side and obtain
\begin{equ}[eq:second]
\mathbb{E} \Vert \bar{v} - v_{\varepsilon} \Vert_{\alpha, t_\varepsilon} \leq C \left(\mathbb{E} \Vert u^{0} - u_{\varepsilon}^{0} \Vert_{\mathcal{C}^{\alpha}} + \varepsilon^{\frac{1}{2} - 3\alpha}\right).
\end{equ}

From the definition of $\bar{u}$ via $\bar{v}$ and \eqref{eq:second} we conclude
\begin{equs}
\mathbb{E} \Vert \bar{u} - u_{\varepsilon} \Vert_{\mathcal{C}^{\alpha}_{t_\varepsilon}} &\leq \mathbb{E} \Vert \bar{v} - v_{\varepsilon} \Vert_{\mathcal{C}^{\alpha}_{t_\varepsilon}} + \mathbb{E} \Vert X - X_{\varepsilon} \Vert_{\mathcal{C}^{\alpha}_{t_\varepsilon}} +  \mathbb{E} \Vert U - U_{\varepsilon} \Vert_{\mathcal{C}^{\alpha}_{t_\varepsilon}}\\
&\leq C \mathbb{E} \Vert u^{0} - u_{\varepsilon}^{0} \Vert_{\mathcal{C}^{\alpha}} + \varepsilon^{\frac{1}{2} - 3\alpha}\;.
\end{equs}
Here, we have also used Lemma \ref{lem:X_X_eps} and the bound
\begin{equs}
\Vert U(t) - U_{\varepsilon}(t) \Vert_{\mathcal{C}^{\alpha}} &\lesssim \Vert u^0 - u_\varepsilon^0 \Vert_{\mathcal{C}^{\alpha}} + \Vert X(0) - X_\varepsilon(0) \Vert_{\mathcal{C}^{\alpha}} \\
&\quad + \varepsilon^{\alpha_\star - 2\alpha} \left(\Vert u_\varepsilon^0 \Vert_{\mathcal{C}^{\alpha_\star}} + \Vert X_\varepsilon(0) \Vert_{\mathcal{C}^{\alpha_\star}} \right),
\end{equs}
which can be derived similarly to \eqref{eq:initial_eps}. The rest of the proof is almost identical to the proof of \cite[Theorem 1.5]{HMW14}.
\end{proof}

\appendix

\section{Regularity of distribution-valued processes}
\label{sec:distributions}

In this section we introduce the Besov spaces and give a Kolmogorov-like criterion for distribution-valued processes to belong to these spaces.

Any distribution $\psi$ defined on the circle $\mathbb{T}$ can be written as the Fourier series
\[
\psi(x) = \frac{1}{\sqrt{2\pi}} \sum_{k \in \intgr} \hat{\psi}(k) e^{ikx}\;.
\]
For $n \geq 1$ we define the $n$th Paley-Littlewood block of $\psi$ as
\[
\delta_n \psi(x) := \frac{1}{\sqrt{2\pi}} \sum_{2^{n-1} \leq |k| < 2^n} \hat{\psi}(k) e^{ikx}\;,
\]
and by definition $\delta_0 \psi \equiv \hat{\psi}(0) / \sqrt{2\pi}$. 

\begin{definition}
For any $\alpha \in \real$, the Besov space $\mathcal{B}^{\alpha}_{\infty, \infty}(\mathbb{T})$ consists of those distributions on $\mathbb{T}$, for which the norm
\[
\Vert \psi \Vert_{\mathcal{B}^{\alpha}_{\infty, \infty}} := \sup_{n \geq 0} 2^{\alpha n} \Vert \delta_n \psi \Vert_{\mathcal{C}^0}
\]
is finite. We denote $\mathcal{C}^\alpha(\mathbb{T}) = \mathcal{B}^{\alpha}_{\infty, \infty}(\mathbb{T})$ for $\alpha < 0$.
\end{definition}

For $\alpha \in (0,1)$ the Besov space $\mathcal{B}^{\alpha}_{\infty, \infty}(\mathbb{T})$ coincides with the H\"{o}lder space $\mathcal{C}^{\alpha}(\mathbb{T})$. The proof of this fact and more information on the Besov spaces can be found in \cite{BCD11}.

For $n \geq 1$ we define the Dirichlet kernel
\begin{equ}
D_n (x) := \frac{1}{\sqrt{2\pi}} \sum_{|k| < 2^n} e^{ikx} = \frac{1}{\sqrt{2\pi}} \frac{\sin\left( \left(2^n - \frac{1}{2} \right) x\right)}{\sin\left( \frac{1}{2}x \right)}\;,
\end{equ}
and $D_0 \equiv 1$.

The following Lemma provides a bound on the Dirichlet kernel $D_n$ in $L^p$ spaces.

\begin{lemma}
\label{lem:D_bound}
For every $1 < p \leq \infty$ there is a constant $C = C(p)$ such that
\[
\Vert D_n \Vert_{L^p(\mathbb{T})} \leq C 2^{\frac{n}{p'}}
\]
holds for every $n \geq 0$, where $p'$ is the conjugate exponent of $p$.
\end{lemma}
\begin{proof}
In the case $p = \infty$, the function can be bounded by its value at $0$, which gives $|D_n(x)| \leq 2^{n+1}$. If $1 < p < \infty$, then we can rewrite
\begin{equs}
\Vert D_n \Vert_{L^p(\mathbb{T})}^p &= \frac{1}{(2\pi)^{p/2}} \int_{-\pi}^\pi \left| \frac{\sin\left( \left(2^n - \frac{1}{2} \right) x\right)}{\sin\left( \frac{1}{2}x \right)} \right|^p dx \\
&= \frac{2^{n(p-1)}}{(2\pi)^{p/2}} \int_{-\pi 2^n}^{\pi 2^n} \left| \frac{\sin\left( \left(1 - 2^{-(n+1)} \right) x\right)}{2^n \sin\left( 2^{-(n+1)}x \right)} \right|^p dx\;.
\end{equs}
The latter integral is bounded by a constant $C(p)$, since the integrand can be estimated up to a constant multiplier by $1 \wedge |x|^{-p}$. That gives the claimed estimate.
\end{proof}

Now, we provide a Kolmogorov-like criterion for distribution-valued processes.

\begin{lemma}
\label{lem:Kolmogorov}
Let $\psi$ be a random field on $[0,T] \times \mathbb{T}$, such that for every $t \in [0,T]$, $\psi(t)$ is a distribution taking values in a fixed Wiener chaos. Furthermore, let us assume that for every $n \geq 0$ the $n$th Paley-Littlewood block satisfies
\begin{equs}
\mathbb{E}\left[ \left| \delta_n \psi(t,x) \right|^2 \right] &\leq A 2^{-2n \alpha}\\
\mathbb{E}\left[ \left| \delta_n \psi(t,x) - \delta_n \psi(s,x) \right|^2 \right] &\leq B 2^{-2 n \alpha} |t-s|^\delta\;,
\end{equs}
for every $x \in \mathbb{T}$, and $t,s \in [0,T]$, and some constants $A, B >0$, $\delta > 0$ and $\alpha < 1$, $\alpha \neq 0$. Then, for any $\gamma < \alpha$, $\gamma \neq 0$, there is a constant $C = C(\alpha, \gamma)$ such that
\begin{equ}
\label{eq:psi_bound_0}
\mathbb{E} \Vert \psi \Vert_{\mathcal{C}_T^{\gamma}} \leq C (A + B)^{\frac{1}{2}}\;.
\end{equ} 
\end{lemma}
\begin{proof}
We can notice that $\delta_n \psi(t,x) = D_n * \delta_n \psi(t,x)$, where the convolution is taken over the variable $x \in \mathbb{T}$. Therefore, the H\"{o}lder inequality yields
\begin{equ}[eq:psi_bound_1]
|\delta_n \psi(t,x)| \leq \Vert D_n \Vert_{L^{p'}(\mathbb{T})} \Vert \delta_n \psi(t) \Vert_{L^p(\mathbb{T})}\;,
\end{equ}
for any $p \geq 1$, where as usual $p'$ is the exponent conjugate of $p$. Since $\psi(t)$ belongs to a fixed Wiener chaos, the same is true for the Paley-Littlewood block $\delta_n \psi(t)$, and we can apply Nelson's lemma to it \cite{Nel73}, saying that all moments of $\delta_n \psi(t)$ are bounded up to a constant multiplier by its second moment. Therefore, 
\begin{equ}[eq:psi_bound_2]
\mathbb{E} \Vert \delta_n \psi(t) \Vert^p_{L^p(\mathbb{T})} \lesssim \int_{\mathbb{T}} \left(\mathbb{E} |\delta_n \psi(t, x)|^2\right)^{\frac{p}{2}}\, dx \lesssim \left(A 2^{-2 n\alpha}\right)^{\frac{p}{2}}\;,
\end{equ}
where the proportionality constant depends on $p$. Combining the bounds \eqref{eq:psi_bound_1}, \eqref{eq:psi_bound_2} together with Lemma \ref{lem:D_bound} and Jensen's inequality, we derive
\begin{equs}
\mathbb{E} \Vert \delta_n \psi(t) \Vert^2_{L^\infty} &\leq \Vert D_n \Vert_{L^{p'}(\mathbb{T})}^2 \mathbb{E} \Vert \delta_n \psi(t) \Vert_{L^p(\mathbb{T})}^2 \\
&\leq \Vert D_n \Vert_{L^{p'}(\mathbb{T})}^2 \left(\mathbb{E} \Vert \delta_n \psi(t) \Vert^p_{L^p(\mathbb{T})}\right)^{\frac{2}{p}} \lesssim A 2^{2 n(\frac{1}{p'} - \alpha)}\;.
\end{equs}
Since for $\gamma < 1$, $\gamma \neq 0$, the space $\mathcal{C}^\gamma$ coincides with the Besov space $\mathcal{B}^\gamma_{\infty, \infty}$, we obtain
\begin{equs}
\mathbb{E} \Vert \psi(t) \Vert^2_{\mathcal{C}^{\gamma}} = \mathbb{E} \Big[\sup_{n \geq 0} 2^{2 n \gamma} \Vert \delta_n \psi(t) \Vert^2_{\mathcal{C}^0} \Big] &\leq \sum_{n \geq 0} 2^{2 n \gamma} \mathbb{E} \Vert \delta_n \psi(t) \Vert^2_{\mathcal{C}^0} \\
&\leq C A \sum_{n \geq 0} 2^{2 n (\gamma + \frac{1}{p'} - \alpha)}\;,
\end{equs}
which is finite if $\gamma < \alpha - \frac{1}{p'}$. Finally, we can notice that for any $\gamma < \alpha$, we can choose $p' \geq 1$ large enough such that $\gamma < \alpha - \frac{1}{p'}$, so that
\begin{equ}
\label{eq:psi_bound_3}
\mathbb{E} \Vert \psi(t) \Vert^2_{\mathcal{C}^{\gamma}} \leq C(\alpha, \gamma) A\;,
\end{equ}
for every $\gamma < \alpha$. Repeating the same argument for $\delta_n \psi(t) - \delta_n \psi(s)$, we derive
\begin{equ}
\label{eq:psi_bound_4}
\mathbb{E} \Vert \psi(t) - \psi(s) \Vert^2_{\mathcal{C}^{\gamma}} \leq C(\alpha, \gamma) B |t-s|^{\delta}\;.
\end{equ}

Since $\psi(t)$ belongs to a fixed Wiener chaos, Nelson's lemma \cite{Nel73} yields equivalence of moments for $\Vert \psi(t) \Vert_{\mathcal{C}^{\gamma}}$ and $\Vert \psi(t) - \psi(s) \Vert_{\mathcal{C}^{\gamma}}$, and we can finish the proof by applying the Banach space-valued version of the Kolmogorov continuity criterion \cite[Lem.~B.3]{HMW14}, which gives the estimate \eqref{eq:psi_bound_0} from \eqref{eq:psi_bound_3} and \eqref{eq:psi_bound_4}.
\end{proof}

The following Lemma provides a bound on the product of two distributions from certain H\"{o}lder spaces.

\begin{lemma}
\label{lem:multipl}
Let $\varphi \in \mathcal{C}^\alpha$ and $\psi \in \mathcal{C}^\beta$, where $\beta < 0 < \alpha < 1$ with $\alpha + \beta > 0$. Then there is a constant $C = C(\alpha, \beta)$ such that 
\[
\Vert \varphi \psi \Vert_{\mathcal{C}^\beta} \leq C \Vert \varphi \Vert_{\mathcal{C}^\alpha} \Vert \psi \Vert_{\mathcal{C}^\beta}\;.
\]
\end{lemma}

The proof of this result can be found in \cite[Theorem 2.85]{BCD11}.


\section{Regularity properties of the semigroups}
\label{appx:heat_sg}

In this appendix we list some properties of the heat semigroup $S_t = e^{t\Delta}$, defined as a convolution on the circle $\mathbb{T}$ with the heat kernel \eqref{eq:heat_kernel}, and the approximate heat semigroup $S^{(\eps)}_t = e^{t\Delta_\varepsilon}$, which is defined as a convolution with the approximate heat kernel \eqref{eq:approx_heat_kernel}. 

The following Lemma provides the regularising property of the heat semigroup $S_t$ in the H\"{o}lder spaces.

\begin{lemma}
\label{lem:S_reg}
Let $\alpha < \beta$, $\beta \geq 0$, then for $t >0$ one has $\Vert S_t \Vert_{\mathcal{C}^\alpha \rightarrow \mathcal{C}^\beta} \lesssim t^{\frac{\alpha - \beta}{2}}$.
\end{lemma}

For $\alpha \leq 0$ and integer $\beta$, one can easily show this bound by the definition of the H\"{o}lder spaces. For non-integer $\beta$ the bound follows by interpolation. A proof of the Lemma for $\alpha \geq 0$ and $\beta \leq \alpha + 1$ can be found in \cite[Lemma 47]{GIP12}. For larger values of $\beta$, the estimate can be shown by using the semigroup property of $S_t$.

The following results provide the regularizing properties of the approximate semigroup $S_\varepsilon$, defined  in the beginning of Section \ref{sec:approx}. All the missing proofs can be found in \cite[Section 6]{HMW14}. 
We assume that Assumption~\ref{as:first} holds in order to derive these bounds.
First, we give a bound on the difference between $S_t$ and $S^{(\eps)}_t$.

\begin{lemma}
\label{lem:S_S_eps}
Let $\lambda \in [0,1]$ and $\alpha \leq \gamma+\lambda$. Then for $\kappa > 0$ sufficiently small and $t > 0$ one has $\Vert S_t - S^{(\eps)}_t \Vert_{\mathcal{C}^{\alpha} \rightarrow \mathcal{C}^{\gamma}} \lesssim t^{-\frac{1}{2}(\gamma - \alpha + \lambda + \kappa)} \varepsilon^{\lambda}$.
\end{lemma}

The following result is analogous to the regularisation property of the heat semigroup.

\begin{lemma}
\label{lem:S_eps}
For any $\gamma, \bar{\gamma} \geq 0$, any $t > 0$ and any $\kappa > 0$ sufficiently small one has $\sup_{\varepsilon \in (0,1)} \Vert S^{(\eps)}_t \Vert_{\mathcal{C}^{\bar{\gamma}} \rightarrow \mathcal{C}^{\bar{\gamma} + \gamma - \kappa}} \lesssim t^{-\frac{\gamma}{2}}$.
\end{lemma}

\endappendix

\bibliographystyle{Martin}
\bibliography{bibliography}

\end{document}